\documentclass[11pt]{article}

\usepackage{amsmath}
\usepackage{amssymb}
\usepackage{amsthm}
\usepackage{latexsym}
\usepackage{color}
\usepackage{graphicx}
\usepackage{color}
\usepackage{caption}
\usepackage{subcaption}
\usepackage[normalem]{ulem}
\usepackage{enumitem}

 \usepackage{xcolor}
 
\DeclareSymbolFont{calletters}{OMS}{cmsy}{m}{n}
\DeclareSymbolFontAlphabet{\mathcal}{calletters}

\def\be{\begin{eqnarray}}
\def\ee{\end{eqnarray}}

\def\b*{\begin{eqnarray*}}
\def\e*{\end{eqnarray*}}

\newtheorem{Theorem}{Theorem}[section]
\newtheorem{Definition}[Theorem]{Definition}
\newtheorem{Proposition}[Theorem]{Proposition}

\newtheorem{Assumption}[Theorem]{Assumption}

\newtheorem{Lemma}[Theorem]{Lemma}

\newtheorem{Remark}[Theorem]{Remark}
\newtheorem{Example}[Theorem]{Example}

\makeatletter \@addtoreset{equation}{section}

\usepackage{color}

\newcommand{\bea}{\begin{eqnarray}}
\newcommand{\bes}{\begin{subequations}}
\newcommand{\ees}{\end{subequations}}
\newcommand{\bgt}{\begin{gather}}
\newcommand{\egt}{\begin{gather}}
\newcommand{\eea}{\end{eqnarray}}
\newcommand{\beaa}{\begin{eqnarray*}}
\newcommand{\eeaa}{\end{eqnarray*}}

\def \E{\mathbb{E}}

\def \P{\mathbb{P}}
\def \Q{\mathbb{Q}}
\def \R{\mathbb{R}}
\def \S{\mathbb{S}}
\def \M{\mathbb{M}}
\def \N{\mathbb{N}}

\def\Ac{{\cal A}}
\def\Bc{{\cal B}}
\def\Cc{{\cal C}}

\def\Lc{{\cal L}}

\def\Cb{\overline{C}}

\def \Om{\Omega}

\def \eps{\varepsilon}

\def \0{\mathbf{0}}

\def \vp{\varphi}
\def \x{\times}

\def\1{\mathbf{1}}
\def\xr{{\rm x}}
\def\vr{{\rm v}}

\def\CT{C([0,T])}

\def\DT{D([0,T])}

 \def\Ninfty#1{\|#1\|}    
 \def\Cb{{\mathbb C}}

  \def\vs#1{\vspace{2mm}}

 \def\Tr{{\rm Tr}}
 \def\bp{{\footnotesize{\boxplus}}}

\title{Approximate viscosity solutions   of path-dependent PDEs and Dupire's vertical differentiability}

\author{
Bruno Bouchard
\footnote{CEREMADE, Université Paris-Dauphine,  PSL, CNRS, 75016 Paris, France.  bouchard@ceremade.dauphine.fr. } 
\and
Gr\'egoire Loeper
\footnote{Monash University, School of Mathematical Sciences \& Centre for Quantitative Finance and Investment Strategies (CQFIS),  gregoire.loeper@monash.edu. 
The CQFIS has been supported by BNP Paribas.}
\footnote{CMAP, Ecole Polytechnique, Palaiseau, France.}
\footnote{BNP Paribas Global Markets.}
\and 
Xiaolu Tan
\footnote{Department of Mathematics, The Chinese University of Hong Kong. xiaolu.tan@cuhk.edu.hk. The research of Xiaolu Tan is supported by CUHK startup grant and CUHK Faculty of Science Direct Grant 2020- 2021.}
}

\begin{document}
\maketitle

\begin{abstract} 
	We introduce a notion of approximate viscosity solution for a class of nonlinear path-dependent PDEs (PPDEs),
	including the Hamilton-Jacobi-Bellman type equations.
	Existence, comparaison and stability results have been established  under fairly general conditions.
	It is also consistent with the notion of smooth solution when the dimension is less or equal to two, or the non-linearity is concave in the second order space derivative.
	We finally investigate the regularity (in the sense of Dupire) of the solution to the PPDE.
\end{abstract}

\section{Introduction}
	  
	The main objective of the paper is to study nonlinear path-dependent PDEs (PPDEs) of the form
	\begin{equation}\label{eq:PPDE_intro}
		-\partial_{t} \vr(t,\xr)
		-
		F\big(t,\xr,\vr(t,\xr),\nabla_{\xr}\vr(t,\xr),\nabla^{2}_{\xr}\vr(t,\xr) \big)
		=
		0,
		~~(t, \xr) \in \Theta,
	\end{equation}
	where $\Theta=[0,T] \x \CT$ or $\Theta = [0,T] \x \DT$,  $\CT$ being the space of all $\R^d$--valued continuous paths on $[0,T]$,
	and $\DT$ being the Skorokhod space of all $\R^d$--valued c\`adl\`ag paths on $[0,T]$. In the above, 
	$\partial_t \vr$ (resp.~$\nabla_{\xr} \vr$ and $\nabla_{\xr}^2 \vr$) represents the horizontal derivative (resp. first and second order vertical derivatives) of $\vr$ 
	in the sense of Dupire \cite{dupireito}, see also Section \ref{subsec:Comparisons} below.
	
	\vspace{0.5em}

	Extending classical PDEs,   PPDEs have various applications in the study of   dynamic stochastic systems  in non-Markovian settings (hedging/super-hedging   of path-dependent options \cite{dupireito},
	  path-dependent optimal control problems and stochastic differential games \cite{pham2014two}, etc.).
	Pioneering works, such as  \cite{dupireito, cont2013functional, peng2016bsde}, investigated  the existence of classical solutions,
	but it is generally very hard to obtain the required regularity of solutions, even for the most simple path-dependent heat equation.
	For this reason, much effort has been devoted to introduce appropriate weak notions of solution.

	\vspace{0.5em}

	In a first stream of literature, see e.g. \cite{ekren2014viscosity, ekren2016viscosity, ekren2016viscosity2, ren2017comparison}, the authors investigate a  relaxed notion of viscosity solutions for PPDEs.
	 As for the classical notion of viscosity solutions for PDEs on $\R^{d}$, see e.g. \cite{CrandallIshiiLions}, they consider families of smooth test functions and use the derivatives of the test functions to define the viscosity supersolution and subsolution properties.
	However, in contrast to the pointwise tangent property used in the classical setting, 
	their definition appeals to an optimality property with respect to a stopping problem under a nonlinear expectation. 
	For fully nonlinear equations, they require $F$ to be non-degenerate  as well as a strong uniform continuity conditions (which for instance are not satisfied in the case of linear equations with path-dependent coefficients).
	In \cite{ren2017comparison}, the non-degeneracy condition has been removed, but more restrictive continuity conditions on $F$ are needed.
	 More recently, \cite{cosso2019crandall}, followed by \cite{zhou2020viscosity, cosso2021path}, introduced a notion of Crandall-Lions viscosity solutions, and investigated the use of Gauge functions for HJB equations.

	\vspace{0.5em}
	
	Another approach consists in approximating the generator $F$ of the PPDE by a sequence $(F^{n})_{n \ge 1}$, for which strong existence holds, and then  define the limit of the corresponding sequence of solutions $(v^{n})_{n\ge 1}$ as a   solution of the original PPDE.
	This is the case of the notion of strong viscosity solution studied in \cite{cossorusso14}, according to which a function $\vr$ is called a strong viscosity solution of the PPDE with generator $F$ 
	if there exists a sequence $(v^{n}, F^{n})_{n \ge 1}$ that converges to  $(\vr, F)$ pointwise, 
	and each $v^{n}$ is a smooth solution to the PPDE with generator $F^{n}$.
	Existence and uniqueness of   strong viscosity solutions was proved for a class of parabolic semilinear PPDEs, which can be expected because of their link with BSDEs and the associated stability results.
	Finally, in \cite{ekren2016pseudo}, the authors introduced a notion of pseudo-Markovian viscosity solutions for fully nonlinear PPDEs by considering a sequence of PPDEs  with generators $F^{n}$ associated to paths frozen up to the exit time of certain domains. 
	In their case, existence in the fully nonlinear setting relies on the convergence of a suitable sequence of functions, which is a-priori not trivial.
	
	\vspace{0.5em}	
	
	Let us also mention the notion of Sobolev solutions investigated in \cite{peng2015g} for a class of PPDEs, which uses global properties to define a solution, in contrast to the local property required in the definition of classical solutions or viscosity solutions.	
	
	\vspace{0.5em}	
	
	In this paper, we   introduce a novel weak notion of solutions in the same vein as the strong viscosity solutions of \cite{cossorusso14} and the pseudo-Markovian viscosity solutions of \cite{ekren2016pseudo}, in the sense that it is defined by an approximation argument.
	Similar to \cite{ekren2016pseudo}, we   consider   PPDEs with fully nonlinear,  and possibly degenerate, generator $F$, which can exhibit a non-uniformly continuous linear part.
	As in \cite{ekren2016pseudo}, our construction is based on a finite dimensional approximation of paths, that are frozen on some time intervals.	
	The difference is that we do not  consider  a sequence of (possible infinitely many) hitting times to define the frozen-path PPDE, 
	but use a simple sequence 
	$\pi = (\pi^n)_{n\ge 1}$ of  discrete time grids on $[0, T]$, with $\pi^{n} = (0 = t^{n}_0 < \cdots < t^{n}_n = T)$,
	and then freeze the path argument on each time interval $[t^{n}_i, t^{n}_{i+1})$ to obtain our sequence of generators $F^{n}$, see \eqref{eq:def_Fn} below. 	Importantly, the path argument is not only frozen, it is also approximated by a piecewise constant or a piecewise linear approximation. For PPDEs having a probabilistic interpretation in terms of BSDEs or optimal control problems, this amounts to replacing  in the coefficients the original diffusion process by the corresponding piecewise constant or   piecewise linear approximation. From this point of view, this definition is very natural and simple. The fact of freezing the path allows us to reduce to a finite dimensional setting, on which standard viscosity solution techniques can be applied. The fact of making an additional approximation of the original path allows us to view the corresponding PDEs as PDEs with a finite number of parameters on which a-priori regularity estimates can be obtained.  
	
	\vspace{0.5em}	
	
	  In this setting, we say that a function $\vr$ is a $\pi$-approximate viscosity solution of the PPDE if it is the limit of the (classical) viscosity solutions $(v^n)_{n \ge 1}$
	associated to the sequence $(F^{n})_{n\ge 1}$, see Definition \ref{def: approximate solution} for a precise statement.
	Without defining explicitly a notion of solution, this approximation approach was already used successfully   in \cite{bouchard2021c} to obtain $\Cb^{1+\alpha}$-type estimates on the candidate solution of a specific class of fully nonlinear PPDEs. We show in this paper that it can actually allow one to consider fairly general situations. 
	
	\vspace{0.5em}	
	
	Using this definition, the uniqueness and comparison are direct consequences of the comparison principle of classical PDEs. 
	As for the existence, it is easy to obtain for a class of semilinear PPDE or HJB type equations, where the solution has a probabilistic representation by the path-dependent BSDEs or optimal control problems.
	In fact, it is generally enough to assume some uniform continuity conditions on the coefficient functions,
	and then to apply the standard stability results of the BSDEs and SDEs, c.f.  Proposition \ref{prop : stoch representation}. 
	However, it becomes more challenging to prove existence in the general nonlinear case, without the help of a probabilistic representation.
	This is our main result, Theorem \ref{thm: existence + Lipschitz cas differentiable}. 
	It is proved in a rather general setting ensuring that comparison and existence holds at the level of the path-frozen finite dimensional PDEs, and only uses viscosity solution techniques.
	In particular, the generator can be degenerate and does not need to be concave  when $d > 2$, an assumption made in  \cite{ekren2016viscosity2}. {Notice also that existence of Crandall-Lions solutions is proved in \cite{cosso2019crandall,cosso2021path,zhou2020viscosity} only for path-dependent HJB equations.}

	\vspace{0.5em}	
		
	 At the same time, we are able to show that a classical solution of the PPDE is an approximate viscosity solution, when $d \le  2$ or the non-linearity is concave (or convex) in the second order spatial derivative. 	In particular, our result could induce natural numerical schemes for the class of nonlinear PPDEs, with an explicit convergence rate, see Remark \ref{rem: vitesse conv vn to v}.

	\vspace{0.5em}	

	Finally, another advantage of our approximation technique is that it allows one to obtain $\Cb^{\beta,1+\alpha}$-type regularity estimates (in the sense of Dupire) for solutions to   PPDEs, a subject which does not seem to have been investigated so far.  In particular, this enables   to  use the $\Cb^{0,1}$-functional It\^o's formula of \cite{bouchard2021c}, see this reference for a specific application in mathematical finance. In this paper, we show that the solution inherits the $C^{1+\alpha}$-regularity of the coefficients and the terminal condition. 
	 For semi-linear equations, the results can actually be proved by purely probabilistic arguments. 
	For fully nonlinear equations, the approximation approach is fully exploited. 
	In this case, we need to assume additional structure conditions, and stay in the context where the generator is concave with respect to the Hessian matrix or the dimension does not exceed $2$. 
	The (very difficult) question of whether a uniform ellipticity condition can regularize the solution is left for future research.

%
%
%

	\vspace{0.5em}	

	 The rest of the paper is organized as follows.
	In Section \ref{sec:ApproxViscoSol}, we introduce our notion of $\pi$-approximate viscosity solutions to the PPDE, and provide some comparison, existence and stability results, in a fairly general fully nonlinear setting. 
	Then, in Section \ref{sec:Regularity}, we provide some first regularity results on the solutions to the PPDE.

\section{Approximate viscosity solution of path-dependent PDEs}
\label{sec:ApproxViscoSol}

	In this section, we first define our notion of approximate viscosity solution, after introducing some general notations. 
	Its very definition implies that comparison holds under standard assumptions. 
	We next show how it is naturally connected to optimal control problems,
	and then provide an existence result as well as a first stability result in a fully nonlinear setting.
	Finally, we show that the notion of approximate viscosity solution is consistent with the notion of classical solution.

\subsection{Preliminaries}

	By the parabolic and local nature of the PPDE  \eqref{eq:PPDE_intro}, it is natural to consider the case where its coefficients are defined on the space $C([0,T])$ of $\R^d$--valued continuous paths on $[0,T]$. A solution   should then be defined on  $[0,T]\x C([0,T])$. On the other hand, when one investigates the vertical regularity of this solution in the sense of Dupire, one needs to consider paths with jumps,	and hence it is more convenient to define it on the space   $\DT$ of $\R^d$--valued c\`adl\`ag paths on $[0,T]$.
 	For this reason, we will consider two cases $\Om = \CT$ or $\Om = \DT$ equipped with two different metrics.

	\vspace{0.5em}
	For both cases of $\Om$, given $\xr \in \Om$, we define its uniform norm by $\Ninfty{\xr} := \sup_{s \in [0,T]}|\xr_s|$,
	and
	$$
		{\xr_{t\wedge}:=(\xr_{t\wedge s})_{s\le T}}\;\mbox{ for }t\le T\;,\;[\xr]^n_{i}:=(\xr_{t^{n}_{j}})_{j = 0, \cdots, i}
		~\mbox{for}~
		i \le n,
	$$
	in which the later is defined with respect to a given increasing sequence of discrete time grids $\pi=(\pi^{n})_{n \ge 1}$, i.e.
	\begin{itemize}	
 		\item $\pi^{n}=(t^{n}_{i})_{0\le i\le n}$
		with $0 = t^{n}_{0} < \cdots < t^{n}_{n} = T$, for each $n \ge 1$,
		\item $\pi^n \subset \pi^{n+1}$, for each $n \ge 1$,
		\item  $|\pi^n| := \max_{i<n} |t^{n}_{i+1}-t^{n}_{i}| \longrightarrow 0$ as $n\longrightarrow \infty$.
	\end{itemize}	
	
	For each $n \ge 1$, we introduce $\eta_n: [0,T] \longrightarrow [0,T]$ and $\eta^+_n : [0,T] \longrightarrow [0,T]$ defined by
	\begin{equation} \label{eq:def_etan}
		\eta_n(t) := \sum_{i=0}^{n-1}  t^n_i  \1_{\{[t^n_i, t^n_{i+1})\}}(t)+ T\1_{\{T\}}(t)
		~~\mbox{and}~~
		\eta^+_n(t) := \sum_{i=0}^{n-1} t^n_{i+1} \1_{\{(t^n_i, t^n_{i+1}]\}}(t).
	\end{equation}
	We next define different metrics $\rho: \Om \x \Om \longrightarrow \R_+ $ as well as different projection operators $\Pi^n: \Om \longrightarrow \Om$, depending whether $\Om = \CT$ or $\Om = \DT$.
	
	\paragraph{Case I: $\Om = \CT$.}
		Let us define the metric $\rho$ and a family of pseudometrics $(\rho_t)_{t \in [0,T]}$ on $\Om$ by
		\begin{equation} \label{eq:def_rhoC}
			\rho := \rho_T,
			~~~
			\rho_t( \xr, \xr') :=  \Ninfty{\xr_{t\wedge} - \xr'_{t \wedge }}.
		\end{equation}
		Further, for each $n \ge 1$, $t \in (t^n_k, t^n_{k+1}]$, $k = 0, \cdots, n-1$, we define $\Pi^n_t : \Om \longrightarrow \Om$ by
		\begin{eqnarray*}
			\Pi^n_t [\xr]_s
			&:=& 
			\sum_{i=0}^{k-1} 
				\1_{\{s \in [t^n_i, t^n_{i+1})\}} \Big( \frac{t^n_{i+1} - s}{t^n_{i+1} - t^n_i} \xr_{t^n_i} + \frac{s - t^n_i }{t^n_{i+1} - t^n_i} \xr_{t^n_{i+1}} \Big) \\
			&&+~
				\1_{\{s \in [t^n_k, t)\}}\Big( \frac{t -  s}{t - t^n_k} \xr_{t^n_k} + \frac{ s - t^n_k }{t - t^n_k} \xr_{t} \Big)
				+ \1_{\{s \in [t,T]\}} \xr_t,
				~~s \in [0,T].
		\end{eqnarray*}
		Namely, $\Pi^n_t[\xr]$ is the linear interpolation of the points $(\xr_{t^n_0}, \cdots, \xr_{t^n_k}, \xr_t)$.
		By convention, let $\Pi^n [\cdot] := \Pi^n_T[ \cdot]$.

	\paragraph{Case II: $\Om = \DT$.}
		Let $\Lambda_t$ denote the collection of all strictly increasing and continuous bijections from $[0,t]$ to $[0,t]$,
		$I_t: [0,t] \longrightarrow [0,t]$ be the identity function,
		and $\mu$ be a finite positive measure on $[0,T]$ with atoms at finitely number of points in $\pi^n$ for some $n \ge 1$.
		We next define the metric $\rho$ and pseudometrics $(\rho_t)_{t \in [0,T]}$ on $\Om$ by
		\begin{equation} \label{eq:def_rhoD}
			\rho := \rho_T,
			~~~
			\rho_t(\xr, \xr') 
			:=
			\inf_{\lambda \in \Lambda_t} \big( \| \lambda - I_t \| + \| \xr_{t\wedge}-  (\xr' \circ \lambda)_{t\wedge} \| \big)
			+\!
			\int_0^t\! \big| \xr_s - \xr'_s \big| \mu(ds).
		\end{equation}
		For each $n \ge 1$, $t \in (t^n_k, t^n_{k+1}]$, $k = 0, \cdots, n-1$, we define $\Pi^n_t : \xr\in \Om \longmapsto \Pi^n_t [\xr]\in  \Om$ as the piecewise constant approximation of $\xr$:
	$$
		\Pi^{n}_t[\xr]_s
			:=
			\sum_{i=0}^{k-1} \1_{\{s \in [t^{n}_{i},t^{n}_{i+1})\}} \xr_{t^{n}_{i}} 
			+
			\1_{\{s \in [t^{n}_{k},t ) \}} \xr_{t^{n}_{k}} 
			+
			\1_{\{s \in [t,T]\}} \xr_{t},
		~\mbox{for all}~ s \in [0,T],
	$$
	and by convention $\Pi^n[\cdot] := \Pi^n_T[\cdot]$.

	\vspace{0.5em}

	Hereafter, we shall consider the two cases simultaneously and only write $(\Omega,\rho)$, 
	unless we need to specialize to $\Omega=D([0,T])$ as in  Section \ref{subsec:Comparisons} and Section \ref{sec:Regularity} below.

	\begin{Remark} \label{remark:separable_Om}
		$\mathrm{(i)}$ In the case that $\Om = D([0,T])$ and without the integral term w.r.t. $\mu$ in \eqref{eq:def_rhoD},
		the metric $\rho$ is exactly the Skorokhod metric.
		Since the map $\xr \longmapsto \int_0^T \xr_t dt$ is not uniformly continuous under the Skorokhod metric, 
		we add the integral term in \eqref{eq:def_rhoD} to make it more general for later uses. 
		At the same time, to ensure that $\Om$ is a separable space under $\rho$, the measure $\mu$ is assumed to have finitely many atoms.
		Moreover, the atoms are assumed to be in $\pi^n$, for some $n \ge 1$, because of our approximation procedure.
		This is  a constraint on the choice of $\mu$, or on the choice of the discrete time grids $\pi$.
		We also notice that, in this case, 
		$$
			\rho(\xr_{t\wedge}, \xr'_{t\wedge}) 
			~\le~
			\rho_t(\xr_{t\wedge}, \xr'_{t\wedge}) 
			~=~
			\rho_t(\xr, \xr') 
			~\le~
			\big(1+\mu([0,t]) \big) ~ \|\xr_{t\wedge} - \xr'_{t \wedge}\|.
		$$
	
		\noindent $\mathrm{(ii)}$ In both cases of the definitions of $(\Om, \rho, \Pi[\cdot])$, it is easy to see that
		$$
			\rho_t \big(\Pi^{n}_t[\xr], \xr_{t\wedge} \big) \longrightarrow 0
			~\mbox{as}~ n \longrightarrow \infty,
			~~\mbox{for all}~
			\xr \in \Om_t := \{ \xr \in \Om ~: \xr = \xr_{t \wedge } \}.
		$$
		In particular, the space $(\Om_t, \rho_t)$ is separable.
	\end{Remark}

\subsection{Definition of  approximate viscosity solutions}

	Let $\S^d$ denote the space of all symmetric $d \x d$-dimensional matrices, and $\Theta := [0,T] \x \Om$. 
	We say that a functional $u: \Theta \longrightarrow \R$ is non-anticipative if $u(t, \xr)=u(t, \xr_{t \wedge})$ for all $(t,\xr)\in \Theta$.
	We are given a generator function $F: [0,T] \x \Om \x \R \x \R^d \x \S^d \longrightarrow \R$, which is   non-anticipative in the sense that
	$F(t, \xr, y, z, \gamma) = F(t, \xr_{t \wedge}, y, z, \gamma)$ for all $(t, \xr, y, z, \gamma) \in [0,T] \x \Om \x \R \x \R^d \x \S^d$,
	and {a terminal data} $g: \Om \longrightarrow \R$.
	Throughout the paper, we   study the following PPDE:
	\begin{equation}\label{eq:PPDE}
		-\partial_{t} \vr(t,\xr)-F \big( t,\xr,\vr(t,\xr),\nabla_{\xr}\vr(t,\xr),\nabla^{2}_{\xr}\vr(t,\xr) \big) = 0,
		~~(t, \xr) \in [0,T) \x \Om,
	\end{equation}
	with terminal condition
	\begin{equation}\label{eq: cond bord}
		\vr(T,\xr)=g(\xr),~~ \xr\in \Om.
	\end{equation}

	Our definition of approximate viscosity solutions is based on the approximations of the path $\xr$, 
	through the operator $\Pi^{n}$ defined above, which allows us to reduce to a finite dimensional setting.
	Let us first introduce some functional spaces.
	Given $L>0$, let 
	$$
		\Cc_{L}
		:=
		\Big\{
			g: \Om \longrightarrow \R ~:
			g~\mbox{ is continuous w.r.t. $\| \cdot\|$ and satisfies}~
			\sup_{\xr \in \Om} \frac{|g(\xr)|}{1 + \|\xr\| } \le L
		\Big\},
	$$
	and let ${\cal LSC}^{n}_{L}$ (resp.~${\cal USC}^{n}_{L}$) denote the collection of non-anticipative functions $f : \Theta \x \R^{d} \longrightarrow \R$
	satisfying, for all $(t,\xr,x)\in [0, T] \x \Om \x \R^{d}$,
	$$
		|f(t,\xr,x)| 
		~\le~ 
		L \big(1+ \|\xr\|+|x| \big)
		~~\mbox{and}~
		f(t, \xr, x) = f \big( t,  \Pi^n_{\eta_n(t)} [\xr], x \big),
	$$
	and such that $(t, x) \longmapsto f(t, \xr, x)$ is  lower-semicontinuous (resp.~upper-semicontinuous) on  each domain $[t^n_i, t^n_{i+1}) \x \R^d$, $i = 0, \cdots, n-1$.
	Then, it is clear that a functional  $f \in {\cal LSC}^{n}_{L}\cup {\cal USC}^{n}_{L}$ can be represented in the form 
		\begin{equation}\label{eq: form f}
		f (t,\xr,x) 
		~=~
		\sum_{i=0}^{n-1}  f_{i} \big( t,[\xr]^{n}_{i},x \big) \1_{[t^{n}_{i},t^{n}_{i+1})}(t) 
		+
		f_n \big([\xr]^n_n, x \big) \1_{\{t^n_n\}}(t) .
	\end{equation}
	 for a sequence of functions $f_i: [t^n_i, t^n_{i+1}) \x (\R^{d})^{i+1}\x \R^{d} \longrightarrow \R$,  $i=0, \cdots, n-1$, together with a map $f_n: (\R^d)^{n+1} \x \R^d \longrightarrow \R$.
	Let
	$$
		\Cc^{n}_{L}
		~:=~
		{\cal LSC}^{n}_{L}\cap {\cal USC}^{n}_{L}.
	$$
	Finally,  we define
	\begin{equation} \label{eq:def_Fn}
		F^n (t, \xr, y, z, \gamma) 
		:=
		F \big(t, \Pi^{n}_{\eta_n(t)}[\xr], y, z, \gamma \big),
		~\mbox{for all}~
		(t, \xr, y, z, \gamma) \in [0,T) \x \Omega  \x \R \x \R^d \x \S^d.
	\end{equation}

	\begin{Definition}[$\pi^n$-viscosity solution] \label{def:pi_n_solution}
		We say that $ u^n \in {\cal LSC}^{n}_{L}$ (resp.~${\cal USC}^{n}_{L}$) of the form \eqref{eq: form f} is a $\pi^{n}$-viscosity supersolution (resp.~subsolution) of \eqref{eq:PPDE} if, 
		for all $\xr \in \Om$ and $i =0, \cdots, n-1$, 
		the map $(t,x)  \longmapsto u^n(t,  \xr,x) $ is a viscosity supersolution (resp.~subsolution) of 
		\begin{align}\label{eq: PDE sur ti ti+1} 
			-\partial_{t} \vp(t,x)-  F^n \big( t, \xr, \vp(t,x),D\vp(t,x),D^{2}\vp(t,x) \big) =0,
			~~\mbox{on}~
			[t^{n}_{i},t^{n}_{i+1})\x  \R^{d},
		\end{align}
		satisfying the boundary condition
		\begin{equation*}
			\liminf_{t \nearrow t^{n}_{i+1},x'\to x} 
			\delta
			\Big(\vp(t, x')-   u^n \big( t^{n}_{i+1}, \Pi^n[\xr] \bp_{t^{n}_{i+1}} x, x \big) \Big)
			\ge
			0,
			~~\mbox{for all}~
			x\in \R^{d}, 
		\end{equation*}
		with $\delta=1$ (resp.~$\delta=-1$) 
		and\footnote{Hereafter, we identify a constant $x\in \R^{d}$ to the constant process $x\1_{[0,T]}$.} 
		$$
		\xr\bp_{t}\xr':=\xr\1_{[0,t)}+\xr' \1_{[t,T]},\; \mbox{ for all } t\le T\mbox{ and } \xr,\xr'\in \Omega.
		$$
		Moreover, we say that $u^n$ is a $\pi^{n}$-viscosity solution of \eqref{eq:PPDE} if it is both a $\pi^{n}$-viscosity supersolution and subsolution of \eqref{eq:PPDE}. 
	\end{Definition}

	In order to ensure the existence and uniqueness of a $\pi^{n}$-viscosity solution in  $\cup_{L'>0}\Cc^{n}_{L'}$, 
	we make the following  standard assumption throughout the paper,
	in which the continuity or Lipschitz continuity property of a functional $\xr \in \Om \longmapsto f(\xr)$ is defined w.r.t. the
	metric $\rho$.

	\begin{Assumption}\label{ass: F}   
		Let $L > 0$ be a fixed positive constant. 

		\vspace{0.5em}

		\noindent $\mathrm{(i)}$ 
		The functional $F$ is non-anticipative, continuous in time, and for all $K>0$
		and $(t,\xr)\in  \Theta$ such that $\|\xr\|\le K$ there exists $C_{K}>0$ such that the map $(y,z, \gamma)\in \R\x \R^{d}\x \S^{d} \longmapsto F(t, \xr, y,z, \gamma)$ is $C_{K}$-Lipschitz continuous. Moreover,   
		$|F(t,\xr,0,0,0)|\le L(1+\|\xr\|)$.

		\vspace{0.5em}

		\noindent $\mathrm{(ii)}$ For all $K>0$, there exists a continuous map $\varpi_{K}:\R_{+} \longrightarrow \R_{+}$ 
		such that $\varpi_{K}(0)=0$ and 
		\begin{equation} \label{eq:cond_Ishii}
			- 
			F\big(t,\xr,y,\alpha(z-z'), \Gamma_1 \big)
			+
			F\big(t,\xr,y,\alpha(z-z'), \Gamma_2 \big)
			~\le~
			 \varpi_{K} \big(\alpha |z-z'|^{2}+|z-z'| \big),
		\end{equation}
		for all $(t,\xr,y,z,z')\in \Theta\x \R\x \R^{d}\x\R^{d}$, $\Gamma_1, \Gamma_2 \in \S^{d}$ and $\alpha>0$ such that $\| \xr\| \le K$ and
		$$
			-3 \alpha \left(\begin{array}{cc} I_{d} & 0\\0 & I_{d} \end{array}\right)
			\le \left(\begin{array}{cc} \Gamma_1 & 0\\0 & -\Gamma_2 \end{array} \right)
			\le
			3 \alpha \left(\begin{array}{cc} I _{d}& -I_{d}\\-I_{d} & I_{d} \end{array}\right),
		$$
		in which $I_{d}$ denotes the $d\x d$-dimensional identity matrix.

		\vspace{0.5em}

		\noindent $\mathrm{(iii)}$ The map $\gamma \longmapsto F(\cdot, \gamma)$ is non-decreasing (for the natural partial order on $\S^{d}$).

		\vspace{0.5em}

		\noindent $\mathrm{(iv)}$ There exists a compact subset $A$ of  $\R^{d'}$, $d'\ge 1$, 
		non-anticipative continuous maps  $\underline \sigma, \overline \sigma: \Theta \x A \longrightarrow \M^{d}$ 
		and $L$-Lipschitz non-anticipative continuous maps $\underline F, \overline F: \Theta\x \R\x \R^{d} \longmapsto \R$ such that
		\begin{itemize}
			\item[\rm (a)] $\xr\in \Om \longmapsto (\underline \sigma, \overline \sigma)(t,\xr,a)$ is $L$-Lipschitz continuous, uniformly in $(t,a)\in [0,T]\x A$.
			\item[\rm (b)]  For all $\gamma \in \S^{d}$, 
			$$
				\underline F+\inf_{a\in A} \frac12 {\rm Tr} \big[\underline \sigma~ \underline \sigma^{\top} (\cdot,a) ~\gamma \big]
				~\le~
				F(\cdot, \gamma)
				~\le~
				\overline F+\sup_{a\in A} \frac12 {\rm Tr} \big[\overline \sigma ~ \overline \sigma^{\top} (\cdot,a) ~\gamma \big].
			$$
		\end{itemize}
	\end{Assumption}

	\begin{Remark}
	{\rm 
		Notice that for PDE \eqref{eq: PDE sur ti ti+1} on $[t^n_i, t^n_{i+1})$, the path $\xr$ is frozen up to $t^n_i$,
		so that $F^n$ does not depend on the space variable $x$.

		\vspace{0.5em}
	
		Item $\mathrm{(i)}$ of Assumption \ref{ass: F} enables us to work on an unbounded domain by exhibiting a suitable penalty function to obtain a comparison principle, see the proof of  Proposition \ref{prop: comparison n-visco} below. The Lipschitz continuity in $(y,z,\gamma)$ is local in $\xr$, which is not an issue here as the path is frozen in $F^n$.

		\vspace{0.5em}
		
		Items $\mathrm{(ii)}$-$\mathrm{(iii)}$ of Assumption \ref{ass: F} are standard to ensure the well-posedness of the PDE as well as the comparison principle (see e.g. \cite{CrandallIshiiLions}).
		In particular, Item $\mathrm{(ii)}$ is essentially used to verify conditions in Ishii's Lemma.
		Nevertheless, since $\xr$ is frozen up to  time $t^n_i$ for PDE \eqref{eq: PDE sur ti ti+1} on $[t^n_i, t^n_{i+1})$, 
		our condition in \eqref{eq:cond_Ishii} is slightly simpler than that in \cite{CrandallIshiiLions}, and can be trivially made   local in $\xr$ with constants $C_K$.
		We also refer to \cite[Corollaries 4.11 and 4.14]{hairer2015loss} for a relaxation of this condition.

		\vspace{0.5em}

		Item $\mathrm{(iv)}$ of Assumption \ref{ass: F} will be used to exhibit a $\pi^{n}$-viscosity supersolution and a $\pi^{n}$-viscosity  subsolution of \eqref{eq:PPDE} with terminal condition $g\in \Cc_{L}$, having linear growth, uniformly in $n\ge 1$, see  \eqref{eq: linear growth unif n approximate sol} below. 
		It is essentially induced by the Lipschitz continuity of $F$, and could be replaced by  other conditions leading to similar a-priori estimates.
	 }
	 \end{Remark}
	
	\begin{Proposition}\label{prop: comparison n-visco} 
		Let Assumption \ref{ass: F} hold true.
		
		\vspace{0.5em}
		
		\noindent $\mathrm{(i)}$ 		
		Fix $n\ge 1$, 
		let $u^n \in  {\cal LSC}^{n}_{L}$ and $v^n \in  {\cal USC}^{n}_{L}$ be respectively $\pi^{n}$-viscosity supersolution and subsolution
		of \eqref{eq:PPDE}	such that $u^n(t^{n}_{n},\cdot)\ge v^n(t^{n}_{n},\cdot)$. 
		Then, $u^n \ge v^n$ on $\Theta\x \R^{d}$. 
		
		\vspace{0.5em}
		
		\noindent $\mathrm{(ii)}$ Fix $g\in \Cc_{L}$ and $n\ge 1$. Then,    \eqref{eq:PPDE}  has a unique $\pi^{n}$-viscosity solution $v^{n}\in \cup_{L'>0}\Cc^{n}_{L'}$ satisfying the terminal condition
		\begin{align}\label{eq: cond teminale vn}
			v^{n} \big(t^{n}_{n},\xr, \xr_{t^n_n} \big)
			=
			v^{n} \big(t^{n}_{n}, \Pi^n[\xr], \xr_{t^n_n} \big)
			=
			g \big( \Pi^n[\xr] \big),
			~~\mbox{for all}~
			\xr\in \Omega.
		\end{align}
		\noindent $\mathrm{(iii)}$ Fix $g\in \Cc_{L}$ and let   $(v^{n})_{n\ge 1}$ be defined as in $\mathrm{(ii)}$ above. Then, there exists $C >0$, that depends only on $L$,  such that 
		\begin{align}\label{eq: linear growth unif n approximate sol}
			\sup_{n\ge 1} |v^{n}(t,\xr, x)|\le  C (1+\|\xr\|+|x|),
			~\mbox{for all}~
			(t,\xr,x) \in \Theta \x \R^{d}.
		\end{align}
	\end{Proposition}
	\proof $\mathrm{(i)}$ Let us $n \ge 1$, $0 \le i \le n-1$, $\xr \in \Om$.
	If $w$ is a viscosity subsolution of \eqref{eq: PDE sur ti ti+1} on $[t^{n}_{i},t^{n}_{i+1})\x \R^{d}$, then Assumption \ref{ass: F}.$\mathrm{(i)}$ ensures that  $(t,x)\in [t^{n}_{i},t^{n}_{i+1})\x \R^{d} \longmapsto w(t,x)-\eps (1+|x|^{2})e^{4Lt}$ is a viscosity subsolution of \eqref{eq: PDE sur ti ti+1} on $[t^{n}_{i},t^{n}_{i+1})\x \R^{d}$ for all $\eps>0$. 
	Combined with Assumption \ref{ass: F}.$\mathrm{(ii)}$, 
	it follows by standard arguments that \eqref{eq: PDE sur ti ti+1} admits a comparison principle among functions having linear growth (see e.g.~\cite[Section 8]{CrandallIshiiLions}).
	Item (i) of Proposition \ref{prop: comparison n-visco} is then proved by backward induction.
	
	\vspace{0.5em}	
	
	\noindent $\mathrm{(ii)-(iii)}$
	For each $n \ge 1$, the uniqueness of the $\pi^{n}$-viscosity solution $v^n$ of \eqref{eq:PPDE} with terminal condition \eqref{eq: cond teminale vn} follows from Item $\mathrm{(i)}$.
	It remains to construct a solution having the (uniform) linear growth property \eqref{eq: linear growth unif n approximate sol}.
	
	Recall that $A$ is a compact subset of $\R^{d'}$ for some $d' \ge 1$, in Assumption \ref{ass: F}.
	Let us consider a probability space equipped with a standard $d$-dimensional Brownian motion $W$ and the corresponding Brownian filtration,
	and denote by $\Ac$ the collection of all predictable processes taking values in $A$. 
	Then, given $\theta:=(t,\xr, x) \in \Theta \x \R^d$ and $\mathfrak{a}\in \Ac$, let $(X^{\theta,\mathfrak{a}},Y^{\theta,\mathfrak{a}},Z^{\theta,\mathfrak{a}})$  be  the unique triplet of $\R^{d}\x \R\x \R^{d}$-adapted processes satisfying 
	$$
		\E \Big[ \big\|X^{\theta,\mathfrak{a}} \big\|^{2}
		+
		\big\|Y^{\theta,\mathfrak{a}} \big\|^{2}
		+
		\int_{t}^{T} \big|Z^{\theta,\mathfrak{a}}_{s} \big|^{2}ds \Big]
		<
		\infty 
	$$
	and that solves the (decoupled) forward-backward SDE
	\begin{align*}
		X^{\theta,\mathfrak{a}}_{s}& 
		= 
		\1_{\{s \in [0,t)\}} \xr_s +
		\1_{\{s \in [t,T]\}} 
		\Big(
			x + \int_{t} ^{  s} \underline\sigma \big(r, \Pi^{n}_{\eta_n(r)}[X^{\theta,\mathfrak{a}}], \mathfrak{a}_{r} \big)dW_{r}
		\Big), 
		~s \in [0,T],\\
		Y^{\theta,\mathfrak{a}}_{s}
		&=
		g \big( \Pi^{n}[X^{\theta,\mathfrak{a}}] \big) 
		+
		\int_{s}^{T} \underline F \big(r, \Pi^{n}_{\eta_n(r)}[X^{\theta,\mathfrak{a}}],Y^{\theta,\mathfrak{a}}_{r},Z^{\theta,\mathfrak{a}}_{r} \big)dr
		-
		\int_{s}^{T}Z^{\theta,\mathfrak{a}}_{r} dW_{r},
		~s \in [t,T],
	\end{align*}
	where we recall that $\eta_n(r) := t^n_i$ for all $r \in [t^n_i, t^n_{i+1})$, $i=0, \cdots, n-1$.
	It follows then from standard arguments (see e.g. \cite{soner2012wellposedness}), combined with Item (iv) of Assumption \ref{ass: F}, that $(t,\xr,x)\in \Theta\x \R^{d} \longmapsto \underline u^{n}(t,\xr,x):=\inf\{Y^{\theta,\mathfrak{a}}_{t}, $ $\mathfrak{a}\in \Ac\}$ is a $\pi^{n}$-viscosity subsolution (actually solution) of \eqref{eq:PPDE} with terminal condition \eqref{eq: cond teminale vn}. 
	At the same time,  using the Lipschitz continuity condition in Item (iv) of Assumption \ref{ass: F},
	it follows by standard estimates on the BSDEs (see e.g. \cite{el1997backward}) that we can find $C >0$, that depends only on $L$, such that 
	$$
		|Y^{\theta,\mathfrak{a}}_{t}|\le C (1+\|\xr\|+|x|),\; \forall\;\mathfrak{a}\in \Ac.
	$$
	Therefore, $\underline u^{n}$ satisfies the linear growth estimate, uniformly in  $n\ge 1$. 
	Similarly, we can exhibit a  $\pi^{n}$-viscosity supersolution  $\overline u^{n}$ of \eqref{eq:PPDE} with terminal condition \eqref{eq: cond teminale vn} that satisfies the same growth estimate. Then, the existence of a $\pi^n$-viscosity solution $v^n$ and the estimate \eqref{eq: linear growth unif n approximate sol} follow from Perron's method, see e.g.~\cite{CrandallIshiiLions}.
	\endproof

	Given the existence result of Proposition \ref{prop: comparison n-visco}, we can now provide our definition of $\pi$-approximate viscosity solutions. 

	\begin{Definition} [$\pi$-approximate viscosity solution] \label{def: approximate solution} 
		Given $g\in \Cc_{L}$, let $v^{n}$ be the unique $\pi^{n}$-viscosity solution of \eqref{eq:PPDE} satisfying the terminal condition \eqref{eq: cond teminale vn}, for each $n \ge 1$. 
		We say that $\vr$ is a $\pi$-approximate viscosity solution of   \eqref{eq:PPDE}-\eqref{eq: cond bord}
		if $(v^{n})_{n\ge 1}$ admits a pointwise limit and
		$$
			\vr(t,\xr)=\lim_{n\to \infty} v^{n} (t,\xr,\xr_{t}),\;\mbox{ for all } (t,\xr)\in \Theta.
		$$
	\end{Definition}

	\begin{Remark}
		For the class of nonlinear PPDE which will be studied below, 
		we will be able to show that the $\pi$-approximate viscosity solution $\vr$ does not depend on the choice of the sequence $\pi = (\pi^n)_{n \ge 1}$ of discrete time grids.
		Nevertheless, we  keep $\pi$ in Definition \ref{def: approximate solution} to leave some flexibility on this notion of solution that could be useful for larger classes of PPDEs.
		
		\vspace{0.5em}
		
		When $\Om = D([0,T])$, the definition of the metric $\rho$ depends on a finite measure $\mu$ which is implicitly related to $\pi$. 
		This is another reason for keeping $\pi$ in the definition of our notion of solution.
	\end{Remark}

\subsection{Comparison principle, existence and stability properties}

	We provide here some results on the comparison, existence and stability of   $\pi$-approximate viscosity solutions.
	For a $\pi^n$-viscosity solution solution $v^n: \Theta \x \R^d \longrightarrow \R$, we will consider it as a functional on $\Theta$ with
	$$
		v^n(t, \xr) ~:=~ v^n(t, \xr, \xr_t),
		~~\mbox{for all}~ (t, \xr) \in \Theta.
	$$
	For later uses, we also consider a modulus of continuity
	$\varpi_{\circ}: \R_+ \longrightarrow \R_+$ (i.e. a non-decreasing function satisfying $\varpi_{\circ}(0)=0$), which is concave
	and satisfies that,  for some constant $c > 0$, $\varpi_{\circ}(x) \ge cx$ for all $x\ge 0$.
	We further define $\varpi_{\circ}': \R_+ \longrightarrow \R_+$ by
	\begin{equation} \label{eq:def_varpi_p}
		\varpi_{\circ}' (x) ~:=~ \sqrt{\varpi_{\circ}(x^2)},~~\mbox{for all}~x \ge 0.
	\end{equation}

\subsubsection{{Comparison of solutions}}

	By the comparison principle results for $\pi^{n}$-viscosity solutions  of Proposition \ref{prop: comparison n-visco},
	it follows immediately a comparison principle for the $\pi$-approximate viscosity solutions, by Definition \ref{def: approximate solution}.
	We state the result below and omit its proof.

	\begin{Proposition}\label{prop: comparison} 
		For $k=1,2$, let $g^{k}\in \Cc_{L}$, $F^{k}: \Theta \x \R\x \R^{d}\x \S^{d} \longrightarrow \R$ satisfy Assumption \ref{ass: F},
		and $\vr^{k}$ be a $\pi$-approximate viscosity solution   of  \eqref{eq:PPDE}-\eqref{eq: cond bord} associated to $(F^{k},g^{k})$ in place of $(F,g)$. 
		Assume that $F^{1} \ge F^{2}$ and that $g^{1}\ge g^{2}$. 
		Then, $\vr^{1}\ge \vr^{2}$ on $\Theta$. 
	\end{Proposition}

\subsubsection{A first existence result for Hamilton-Jacobi-Bellman equations}

	When PPDE \eqref{eq:PPDE} is in form of a path-dependent Hamilton-Jacobi-Bellman (HJB) equation,
	then, for each $n \ge 1$, its $\pi^n$-viscosity solution $v^n$ is the solution to a classical HJB equation,
	which corresponds to the value function of a controlled diffusion processes problem in which the path of the controlled forward process $X$ is replaced by its approximation $\Pi^{n}[X]$ in the coefficients.
	In this case, the convergence of $(v^n)_{n\ge 1}$ follows from the convergence of $\Pi^{n}[X]$ to $X$.
	Let us provide the following example to illustrate the idea,
	where the continuity or Lipschitz continuity of a functional w.r.t.~$\xr$ is under $\rho$.
	Another general case will be studied later in  Theorem \ref{thm: existence + Lipschitz cas differentiable}.

	\begin{Proposition} \label{prop : stoch representation}
		Let $g\in \Cc_{L}$, and for all $(t,\xr,y,z, \gamma)\in \Theta\x \R\x \R^{d}\x \S^{d}$,
		\begin{align*}
			F(t,\xr,y,z,\gamma) ~=~  \sup_{a\in A} \Big( F_{\circ} (t,\xr,y,z, a) + \frac12 {\rm Tr}[ \sigma \sigma^{\top}(t,\xr,a)  \gamma] \Big),
		\end{align*}
		for some  Borel set $A\subset \R^{d'}$, $d'\ge 1$, and continuous functionals $F_{\circ}: \Theta\x \R\x \R^{d}\x A \longrightarrow \R$ and $\sigma: \Theta\x A \longrightarrow \R$
		such that $\sup_{(t,a)\in [0,T]\x A} \big| F(t, \mathbf{0}, 0, 0, a) \big| < \infty$ and
		\begin{align*} 
			&\sup_{(t, \xr, a)\in \Theta \x A} \frac{|\sigma(t, \xr, a) |}{1+ |\xr|} < \infty,
			~~~
			\xr \longmapsto \sigma(t, \xr, a) ~\mbox{is Lipschitz, uniformly in}~(t, a)\in [0,T]\x A.
		\end{align*}
		Moreover, assume that, for all $(t, \xr, \xr', y, y', z, z', a)\in [0,T]\x \Omega^{2}\x \R^{2}\x (\R^{d})^{2}\x A$,
		$$
			\big| F_{\circ}(t, \xr, y, z, a) - F_{\circ}(t, \xr, y', z', a) \big| 
			~\le~
			L \big(|y-y'| + |z - z'| \big),
		$$
		and
		$$
			\big| g(\xr) - g(\xr') \big| 
			+
			\big| F_{\circ}(t, \xr, y, z, a) - F_{\circ}(t, \xr', y, z, a) \big|
			~\le~
			L  \varpi_{\circ}' \big( \rho_t(\xr, \xr') \big).
		$$
		Then,  PPDE \eqref{eq:PPDE}-\eqref{eq: cond bord} has a unique $\pi^n$-viscosity solution $v^n$ for each $n \ge 1$, and $v^n \longrightarrow \vr$ pointwise for some non-anticipative map $\vr: \Theta \longrightarrow \R$.
		In particular, $\vr$ is a $\pi$-approximate viscosity solution of \eqref{eq:PPDE}-\eqref{eq: cond bord}.
	\end{Proposition}
	\proof Let us fix $(t, \xr) \in \Theta$.
	First, by Proposition \ref{prop: comparison n-visco}, there exists a unique $\pi^{n}$-viscosity solution $v^n$ to  \eqref{eq:PPDE}  with linear growth.
	Moreover, by the representation theorem of Markovian 2BSDE (see e.g. \cite{soner2012wellposedness}), one has $v^n(t, \xr) = \sup_{\mathfrak{a} \in \Ac} Y^{n, \mathfrak{a}}_t$,
	where $(X^{n, \mathfrak{a}}, Y^{n, \mathfrak{a}}, Z^{n, \mathfrak{a}})$ is the unique solution to the FBSDE
	\begin{align*}
		X^{n,\mathfrak{a}}_{s} & = \xr_{s \wedge t} + \int_{t}^{s \vee t}  \sigma\big(r, \Pi^n_{\eta_n(r)}\big[ X^{n,\mathfrak{a}}\big], \mathfrak{a}_{r} \big) dW_{r}, ~s \in [0,T],\\
		Y^{n,\mathfrak{a}}_{s}
		&=
		g \big( \Pi^n\big[ X^{n,\mathfrak{a}} \big] \big) 
		+
		\int_{s}^{T}  F_{\circ} \big(r,  \Pi^n_{\eta_n(r)} \big[ X^{n,\mathfrak{a}} \big] ,Y^{n,\mathfrak{a}}_{r}, Z^{n,\mathfrak{a}}_{r} \big)dr 
		-
		\int_{s}^{T}Z^{n,\mathfrak{a}}_{r} dW_{r}, ~s \in [t,T],
	\end{align*}	
	and $\Ac,$ $W$ are as in  Step (ii)-(iii) of the proof of Proposition \ref{prop: comparison n-visco}.
	Let $X^{\mathfrak{a}}$ be the controlled diffusion process defined by
	$$
		X^{\mathfrak{a}}_{s} ~=~ \xr_{s \wedge t} + \int_{t}^{s \vee t} \sigma\big(r, X^{\mathfrak{a}} ,\mathfrak{a}_{r} \big) dW_{r},
		~s \in [0,T].
	$$
	It follows from standard stability result for SDEs that
	$$
		\sup_{\mathfrak{a} \in \Ac} \E \Big[ \big \| X^{n, \mathfrak{a}} - X^{\mathfrak{a}}  \big\|^2 \Big] 
		~\longrightarrow~0,
		~\mbox{as}~n \longrightarrow \infty.
	$$
	Using the concavity  of the continuity modulus function $\varpi_{\circ}$, one can deduce that
	$$
		\sup_{\mathfrak{a} \in \Ac} 
		\E \Big[ 
			\!\int_t^T\!  \Big| F_{\circ} \big(r, \Pi^n_{\eta_n(r)} [ X^{n,\mathfrak{a}} ], 0,0 \big) - F_{\circ} \big(r, X^{\mathfrak{a}}, 0,0 \big) \Big|^2 dr
			+ \Big| g \big( \Pi^n \big[ X^{n, \mathfrak{a}} \big] \big) - g \big(X^{\mathfrak{a}} \big) \Big|^2
		\Big] 
		\longrightarrow0.
	$$
	Let $(Y^{\mathfrak{a}}, Z^{\mathfrak{a}})$ be the unique solution of the BSDE
	$$
		Y^{\mathfrak{a}}_{s}
		~=~
		g \big( X^{\mathfrak{a}}  \big) 
		+
		\int_{s}^{T}  F_{\circ} \big(r,  X^{\mathfrak{a}}  ,Y^{\mathfrak{a}}_{r}, Z^{\mathfrak{a}}_{r} \big)dr 
		-
		\int_{s}^{T}Z^{\mathfrak{a}}_{r} dW_{r}, ~s \in [t,T].
	$$
	It follows then by standard stability result for BSDEs (see e.g. \cite{el1997backward}) that
	$$
		\sup_{\mathfrak{a} \in \Ac} \big| Y^{n, \mathfrak{a}}_t - Y^{\mathfrak{a}}_t \big| 
		~\longrightarrow~ 0.
	$$
	Therefore, one has
	$$
		v^n(t, \xr) = v^n(t, \xr, \xr_t) = \sup_{\mathfrak{a} \in \Ac} Y^{n, \mathfrak{a}}_t
		~\longrightarrow~
		\vr(t, \xr) := \sup_{\mathfrak{a} \in \Ac} Y^{\mathfrak{a}}_t,
	$$
	and hence $\vr$ is a $\pi$-approximate viscosity solution of   \eqref{eq:PPDE}-\eqref{eq: cond bord} by Definition \ref{def: approximate solution}.
	\qed

	\begin{Remark} \label{rem:existence1}
		$\mathrm{(i)}$ Similarly, one can   consider Hamilton-Jacobi-Bellman-Isaac equations associated to zero-sum games. 
		The critical point is the uniform convergence of the approximation of the controlled diffusion processes as well as the stability of the related BSDEs.
	
		\vspace{0.5em}
	
		\noindent $\mathrm{(ii)}$
		Notice that, in the proof of Proposition \ref{prop : stoch representation},   convergence holds (to the same function) for any sequence $\pi$ of time grids,
		and the limit does not depends on $\pi$. 
		In other words, the exact sequence $\pi$ of time grids entering in Definition  \ref{def: approximate solution} does not play any role.
		This will also be the case for our general existence result in Theorem \ref{thm: existence + Lipschitz cas differentiable} below.
	\end{Remark}

\subsubsection{Existence and stability   for fully nonlinear equations}
\label{subsec existence et stability}

	In this part, we prove  the existence of a unique $\pi$-approximate viscosity solution as well as a stability result  
	for a general class of fully nonlinear equations satisfying the following structure condition. 
	Recall that $\rho$ is defined in \eqref{eq:def_rhoC} and \eqref{eq:def_rhoD} for the two cases of $\Om$,
	the modulus of continuity $\varpi_{\circ}: \R_+ \longrightarrow \R_+$ is concave
	and $\varpi_{\circ}' (x) := \sqrt{\varpi_{\circ}(x^2)}$, for all $x \ge 0$.

	\begin{Assumption} \label{ass:FH}
		There exists a continuous non-anticipative function $H: \Theta\x \R \x \R^{d}\x \S^{d}  \longrightarrow \R$, 
		together with non-anticipative maps $r : \Theta \longrightarrow \R$, $\mu:\Theta \longrightarrow \R$ and $\sigma :\Theta \longrightarrow \S^{d}$,
		such that $\gamma \longmapsto H(\cdot, \gamma)$ is increasing, and for all $(t, \xr, y, z, \gamma)\in [0,T]\x \Om\x \R\x \R^{d}\x \S^{d}$,
		$$
			F(t, \xr, y, z, \gamma) 
			~=~
			H(t, \xr, y, z, \gamma) + r(t,\xr) y + \mu(t,\xr) \cdot z +\frac12 \mathrm{Tr} \big[ \sigma\sigma^{\top} (t,\xr) \gamma \big].
		$$
		Moreover, there is a constant $L>0$ such that for all  $t \in [0,T]$ and $(\xr,y, z, \gamma),(\xr',y', z', \gamma')\in \Om \x \R\x \R^{d}\x \S^{d}$,
		one has $| r(t, \xr) |\le L$ and
		\begin{align}
			&\big| r(t,\xr)-r(t,\xr') \big| + \big| \mu(t,\xr)-\mu(t,\xr') \big| + \big| \sigma(t,\xr)-\sigma(t,\xr') \big| 
				~\le~  L \rho_t(\xr, \xr'),
			\label{eq: hyp lip rho mu sig g}\\
			&\big| H(t,\xr,y, z, \gamma)-H(t,\xr, y', z', \gamma') \big| 
				~\le~ L \big( |y-y'| + |z-z'| + |\gamma- \gamma'| \big),
			\label{eq: hyp lip H}\\
			&\big| g(\xr)-g(\xr') \big|  + \big| H(t,\xr,y, z, \gamma)-H(t,\xr', y, z, \gamma) \big|
				~\le~ L \varpi_{\circ}' \big(\rho(\xr, \xr') \big).
			\label{eq: hyp borne rho}
		\end{align}
	\end{Assumption}

	The above turns out to be enough not only to prove existence and uniqueness but also that the solution does not depend on the particular sequence $\pi = (\pi^n)_{n \ge 1}$ of  time grids used in Definition \ref{def: approximate solution} (see also Remark \ref{rem:existence1}).

	\begin{Theorem}\label{thm: existence + Lipschitz cas differentiable} 
		Let Assumptions \ref{ass: F}  and \ref{ass:FH} hold true.
		Then: 
		
		\vspace{0.5em}
		
		\noindent $\mathrm{(i)}$ PPDE \eqref{eq:PPDE}-\eqref{eq: cond bord} has a  unique $\pi$-approximate viscosity solution $\vr$. 
		
		\vspace{0.5em}

		\noindent $\mathrm{(ii)}$ For all $K > 0$, there exists a constant $C_{K}>0$, that depends only on $K$ and $L$, such that
		\begin{equation}\label{eq: approximate sol is lipschitz cas diff}
			|\vr(t,\xr')-\vr(t,\xr)|\le C_{K}~  \varpi_{\circ}' \big( \rho_t \big(\xr, \xr' \big) \big)
			~~\mbox{and}~~
			|\vr(t',\xr_{ t\wedge})-\vr(t,\xr)|\le  C_{K} \varpi_{\circ}' \big( |t'-t|^{\frac{1}2} \big),
		\end{equation}
		for all $0 \le t\le t' \le T$ and $  \xr,\xr'\in \Om$ satisfying  $\|\xr\|\vee \|\xr'\|\le K$.
		
		\vspace{0.5em}

		\noindent $\mathrm{(iii)}$  If $\pi'$ is  another increasing sequence of discrete time grids and $\vr'$ is the $\pi'$-approximate viscosity solution of  \eqref{eq:PPDE}-\eqref{eq: cond bord}, 
			then $\vr'=\vr$ on $[0,T]\x \Om$.
	\end{Theorem}

	\begin{Remark} \label{rem: vitesse conv vn to v}
		In the context of Theorem \ref{thm: existence + Lipschitz cas differentiable}, let $\vr$ be the $\pi$-approximate viscosity solution and $v^n$ be the $\pi^n$-viscosity solution of \eqref{eq:PPDE}-\eqref{eq: cond bord}. 
		As a by-product of  \eqref{eq:diff_vn_v2n} in the proof below, for all $K>0$, there exists a constant $C_{K}>0$ such that 
		\begin{align}\label{eq: vitesse conv vn cas diff}
			\big| v^{n}(t,\xr,\xr_{t})-   \vr(t,\xr) \big|
			~\le~
			C_{K} ~ 
			\varpi_{\circ}' \Big( \rho_{\eta^+_n(t) } \big(\Pi^n [ \xr \bp_{\eta^+_n(t) } \xr_t ], \xr \big) + |\pi^n|^{\frac14} \Big),
		\end{align}
		whenever $ \|\xr\| \le K$ (recall that $\eta^+_n(t)$ is defined in \eqref{eq:def_etan}).
	\end{Remark}

	Before to prove the above, let us state immediately the following stability result.

	\begin{Proposition}\label{prop : stability}
		Let $(F_{k},g_{k})_{k \ge 0}$ be a sequence satisfying the assumptions in Theorem  \ref{thm: existence + Lipschitz cas differentiable}, 
		uniformly in $k\ge 0$, and such that 
		$$
			\big(F_{k}(t_{k},\xr_{k}, y_{k}, z_{k}, \gamma_{k}),g_{k}(\xr_{k}) \big)
			~\longrightarrow~
			\big(F_0 (t ,\xr, y, z, \gamma), g_0(\xr) \big) 
			~~\mbox{as}~ k \longrightarrow \infty,
		$$
		whenever $(t_{k},\xr_{k}, y_{k}, z_{k}, \gamma_{k}) \longrightarrow (t ,\xr, y, z, \gamma)\in \Theta\x \R\x \R^{d}\x \S^{d}$. 
		Let $\vr_{k}$ be a $\pi$-approximate viscosity solution associated to $(F_{k},g_{k})$ for each $k \ge 0$. 
		Then,  there exists a subsequence $(k_{m})_{m\ge 1}$ such that $(\vr_{k_{m}})_{m\ge 1}$ converges {pointwise} to $\vr_0$.
	\end{Proposition}
	\proof For each $k \ge 0$, let $v^{n}_{k}$ be the $\pi_{n}$-viscosity solution associated to $(F_{k},g_{k})$. 
	Then, for each fixed $n \ge 1$, it follows by stability  of the viscosity solutions of classical PDEs that $ v^n_k  \longrightarrow v^n_0$ pointwise as $k\longrightarrow \infty$.
	
	\vspace{0.5em}

	In view of Remark \ref{remark:separable_Om}, one can find a countable subset $\Theta_{\circ}$ of $\Theta$ 
	such that, for any $(t, \xr) \in \Theta$,
	there exists a sequence $(t_i, \xr_i)_{i \ge 1} \subset \Theta_{\circ}$ satisfying
	\begin{equation} \label{eq:density_Theta_0}
		 \rho_{t_i} \big( \xr_i, \xr \big) + |t_i - t|^{1/2} ~\longrightarrow~ 0.
	\end{equation}
	Then, by a standard diagonalization argument, we can find a subsequence $(k_{n})_{n\ge 1}$ such that $|v^{n}_{k_{n}}-v^{n}_0| \longrightarrow 0$ {pointwise} on the countable set $\Theta_{\circ}$.
	On the other hand, \eqref{eq: vitesse conv vn cas diff} implies that, for each $(t, \xr) \in \Theta_{\circ}$, there is some constant $C_K > 0$ such that, for all $n\ge 1$,
	\begin{align*}
		|\vr_{k_{n}}-\vr_0| (t,\xr)
		~\le~
		|v^{n}_{k_{n}}-v^{n}_0|(t,\xr,\xr_{t})
		+
		2 C_K
		\varpi_{\circ}' \Big( \rho_{\eta^+_n(t) } \big(\Pi^n [ \xr \bp_{\eta^+_n(t) } \xr_t ], \xr \big) + |\pi^n|^{{\frac14}} \Big).
	\end{align*}
	Therefore, by Remark \ref{remark:separable_Om}, one has $\vr_{k_{n}} \longrightarrow \vr_0$ pointwise on $\Theta_{\circ}$.
	Finally, by \eqref{eq: approximate sol is lipschitz cas diff} and \eqref{eq:density_Theta_0}, it follows that $\vr_{k_{n}}\longrightarrow \vr_0$ {pointwise} on $\Theta$. 
	\endproof

	\vspace{0.5em}

	We now complete the proof of Theorem \ref{thm: existence + Lipschitz cas differentiable}. The key ingredient is to provide a uniform estimate on the difference $v^m- v^n$ when $m, n \longrightarrow \infty$,
	where $v^m$ (resp. $v^n$) is the $\pi^{m}$-(resp. $\pi^{n}$-)viscosity solution of \eqref{eq:PPDE} with terminal condition \eqref{eq: cond teminale vn}.
	
	\vspace{0.5em}
	
	By \eqref{eq: form f} and Definition \ref{def:pi_n_solution}, it is clear that $v^n(t, \xr, x) $ depends only on $(\xr_{t^n_j})_{j=0, \cdots, i}$ and $x$ when $t \in [t^n_i, t^n_{i+1})$.
	For this reason, for each $n \ge 1$ and $t \in [t^n_i, t^n_{i+1})$, we introduce $v^n_i: [t^n_i, t^n_{i+1}) \x (\R^d)^{i+1} \x \R^d \longrightarrow \R$ defined by
	\begin{align}\label{eq: def vni} 
		v^{n}_{i} \big(t,[\xr]^{n}_{i},x \big)
		~:=~
		v^{n}\big( t, \Pi^n_{t^{n}_{i}} [\xr], x \big),
	\end{align}
	and $F^{n}_{i} : [t^{n}_{i},t^{n}_{i+1})\x (\R^{d})^{i+1}\x \R^{d} \longrightarrow \R$ as well as $g^{n} :(\R^{d})^{n+1} \longrightarrow \R$ defined as 
	\begin{equation} \label{eq:def_Fngn}
		F^{n}_{i} \big( t,[\xr]^{n}_{i},\cdot \big)
		~:=~
		F \big( t, \Pi^n_{t^{n}_{i}} [\xr],\cdot \big),
		~~~
		g^{n} \big([\xr]^{n}_{ n} \big)
		~:=~
		g \big( \Pi^n[\xr] \big).
	\end{equation}
	We similarly define $r^n_i$, $\mu^n_i$, $\sigma^n_i$, i.e.
	$$
		r^{n}_{i} \big( t,[\xr]^{n}_{i} \big)
		:=
		r \big(t, \Pi^n_{t^{n}_{i}} [\xr] \big),
		~~~
		\mu^{n}_{i} \big(t,[\xr]^{n}_{i} \big)
		:=
		\mu \big(t, \Pi^n_{t^{n}_{i}} [\xr] \big),
		~~~
		\sigma^{n}_{i} \big(t,[\xr]^{n}_{i} \big)
		:=
		\sigma \big(t, \Pi^n_{t^{n}_{i}} [\xr] \big).
	$$
	Then, by Definition \ref{def:pi_n_solution}, $(t,x) \longmapsto v^n_i(t, [\xr]^n_i, x)$ is the unique viscosity solution of the classical PDE
	\begin{equation} \label{eq:PDE_vni}
		- \partial_t v^n_i \big(t, [\xr]^n_i, x \big)  - F^n_i \big( t, [\xr]^n_i, v^n_i, D v^n_i, D^2 v^n_i \big) = 0,
		~~(t,x) \in [t^n_i, t^n_{i+1}) \in \R^d,
	\end{equation}
	with terminal condition 
	$$
		v^n_i \big(t^{n}_{i+1}, [\xr]^n_i, x \big) = v^n_{i+1} \big(t^{n}_{i+1}, {([\xr]^n_i, x)}, x \big),
		~
		i=0, \cdots, n-2,
		~~~
		v^n_n \big(t^n_{n}, [\xr]^n_{n} \big) = g^n \big( [\xr]^n_n \big),
	$$
	where $D v^n_i$ (resp. $D^2 v^n_i$) represents the gradient (resp.~Hessian matrix) of $x \longmapsto v^n_i(t, [\xr]^n_i, x)$.

	We first study the stability of $(v^{n})_{n\ge 1}$ with respect to the space arguments. 
	
	\begin{Lemma} \label{lem:Unif_conti_vni}
		Let the conditions of Theorem \ref{thm: existence + Lipschitz cas differentiable}   hold,
		and let $v^n$ be the $\pi^{n}$-viscosity solution of \eqref{eq:PPDE} satisfying the terminal condition \eqref{eq: cond teminale vn} for each $n \ge 1$.
		Then, for all $K > 0$,
		there exists a constant $C_K > 0$, such that, for all $m, n \ge 1$, $t \in [0,T]$, and $\xr, \xr' \in \Om$ satisfying $\|\xr_{t \wedge} \| + \|\xr'_{t \wedge}\| \le K$,
		\begin{align} \label{eq: regu en espace de vn}
			\big| v^m(t,\xr, \xr_t)- v^n(t,\xr', \xr'_t) \big| 
			\le &
			C_K \varpi_{\circ}' \Big( \rho_{\eta^+_m(t) } \big(\Pi^n [ \xr \bp_{\eta^+_n(t) } \xr_t ], \Pi^m [ \xr' \bp_{\eta^+_m(t) } \xr_t'] \big) 
			\!+\! 
			|\pi^m|^{\frac14}  
			\!+\!
			|\pi^n|^{\frac14} \Big).
		\end{align}
	\end{Lemma}
	\proof We restrict to the one dimensional case $d=1$ for ease of notations. 
	Given $m \le n$, as $\pi^m \subset \pi^n$, the vector $[\xr]^n_n$ contains all the information of $[\xr]^m_m$.
	We will then consider a functional of $[\xr]^m_{\cdot}$ as a functional of $[\xr]^n_{\cdot}$.
	For this purpose, let us introduce, for all $i \le n-1$, $t\in [t^{n}_{i},t^{n}_{i+1})$, $\xr \in \Om$ and $x \in \R$,
	\begin{equation} \label{eq:def_vmn}
		v^{m,n}_i \big(t, [\xr]^n_i, x \big)
		~:=~
		v^m_{I^n_m(i)} \big( t, [\xr]^m_{I^n_m(i)}, x \big),
	\end{equation}
	where $I^n_m(i) := \max\{j ~: t^m_j \le t^n_i \}$ for all  $i = 0, 1, \cdots, n$.
	Similarly, one defines the functionals
	$$
		F^{m,n}_i, ~ H^{m,n}_i,~ r^{m,n}_i, ~\mu^{m,n}_i, ~\sigma^{m,n}_i,~g^{m,n}.
	$$

	\noindent $\mathrm{(i)}$
	Given $i<n$, $(t,x,x')\in [t^{n}_{i},t^{n}_{i+1})\x\R^{2}$ and $\xr,\xr'\in \Omega$, set  
	$$
		w^{m,n}_i \big( t,[\xr]^{n}_{i}, [\xr']^{n}_{i}, x, x' \big)
		~:=~
		v^{n}_{i} \big( t,   [\xr]^{n}_{i},x  \big) - v^{m,n}_{i} \big( t,[\xr']^{n}_{i},x' \big).
	$$
	We first show that $ w^{m,n}_i(\cdot,[\xr]^{n}_{i}, [\xr']^{n}_{i}, \cdot,\cdot) $ is a viscosity subsolution of some HJB equation. 
	Let $\phi : [t_{i}^{n},t^{n}_{i+1}]\x \R^{2} \longrightarrow \R$ be a smooth function and $(\hat t,\hat x,\hat x')\in [t_{i}^{n},t^{n}_{i+1})\x \R^{2}$ be such that 
	\begin{align*}
		0~=~ &
		\max_{(t,x,x')\in [t_{i}^{n},t^{n}_{i+1})\x \R^{2}}\left(w^{m,n}_i \big( t,[\xr]^{n}_{i}, [\xr']^{n}_{i}, x, x' \big)-\phi(t,x,x')\right)\\
		~=~&
		w^{m,n}_i \big( \hat t,[\xr]^{n}_{i}, [\xr']^{n}_{i}, \hat x, \hat x' \big)-\phi(\hat t,\hat x,\hat x').
	\end{align*}
	Recall that $v^{n}_{i} $ is a viscosity solution (hence supersolution) of  \eqref{eq: PDE sur ti ti+1} with generator $F^n$
	and $v^{m,n}_i$ is a viscosity solution (hence subsolution) of  \eqref{eq: PDE sur ti ti+1} with generator $F^{m,n}$.
	Then, it follows from Ishii's lemma (see e.g. \cite[Theorem 3.2]{CrandallIshiiLions}) that,
	for all $\eps>0$, one can find $\gamma_{\eps}, \gamma'_{\eps} \in \R$ such that 
	\begin{align}
		\left(\begin{array}{cc}
		\gamma_{\eps}& 0\\
		0 & -\gamma'_{\eps}
		\end{array}\right)&\le D^{2} \phi(\hat t,\hat x,\hat x')+\eps \left(D^{2} \phi(\hat t,\hat x,\hat x')\right)^{2},
		\label{eq: inega IShii}
	\end{align}
	and 
	\begin{align}
		0 ~\ge~ & 
		-~ \partial_{t} \phi(\hat t,\hat x,\hat x')
		-
		F^{n} \big( \hat t,    [\xr]^{n}_{i}, v^{n}_{i} \big( \hat t,    [\xr]^{n}_{i},\hat x  \big) , \partial_{x}\phi(\hat t,\hat x,\hat x'), \gamma_{\eps} \big) \nonumber
		\\
		~&~+~
		F^{m,n} \big( \hat t,  [\xr]^{n}_{i} , v^{m,n}_{i} \big( \hat t, [\xr']^{n}_{i} ,\hat x'  \big) , -\partial_{x'}\phi(\hat t,\hat x,\hat x'), \gamma'_{\eps} \big).
		\label{eq: diff Hnm nega}
	\end{align}
	 We now estimate the r.h.s.~of \eqref{eq: diff Hnm nega}.
	By \eqref{eq: hyp lip H} and \eqref{eq: inega IShii}, {using the notation}  $\theta = (\hat t, [\xr]^n_i, \hat x, \hat x', \gamma_{\eps}, \gamma'_{\eps})$,
	\begin{align*}  
		\Delta H^{m,n}_{{\eps}}
		~:=~&
		-~H^{n} \big( \hat t,   [\xr]^{n}_{i}, v^{n}_{i} \big( \hat t,    [\xr]^{n}_{i},\hat x  \big) , \partial_{x}\phi(\hat t,\hat x,\hat x'), \gamma_{\eps} \big) \\
		&+~
		H^{m,n} \big(\hat t,  [\xr']^{n}_{i} , v^{m,n}_{i} \big( \hat t, [\xr']^{n}_{i} ,\hat x'  \big) , -\partial_{x'}\phi(\hat t,\hat x,\hat x'), \gamma'_{\eps} \big)
		\\
		~\ge~& 
		-~ L \Big( \varpi_{\circ}' \big( \rho_{\hat t} \big( \Pi^{n}_{t^n_i} [\xr] ,\Pi^{m}_{t^n_i} [\xr'] \big) \big)
			+  \big| \phi(\hat t,\hat x,\hat x') \big|
			+ \big|\partial_{x}\phi(\hat t,\hat x,\hat x')+\partial_{x'}\phi(\hat t,\hat x,\hat x') \big| 
		\Big)
		\\
		&-~ \partial_{\gamma}H^{n}(\theta) \Big(\Delta \phi(\hat t,\hat x,\hat x')+\eps \Delta^{2} \phi(\hat t,\hat x,\hat x')  \Big),
	\end{align*}
	for some  functional $\partial_{\gamma} H^n\ge 0$ bounded by $L$, and   where 
	$$
		\Delta \phi :=\partial^{2}_{xx}\phi +2 \partial^{2}_{xx'}\phi +\partial^{2}_{x'x'}\phi
		~~\mbox{and}~~
		\Delta^{2} \phi:=(1,1) \left(D^{2} \phi\right)^{2}\left(\begin{array}{cc}
		1\\
		1
		\end{array}\right).
	$$
	This implies that, with $B:= [-L,L]^2 \x [0,\sqrt{L}]$,
	\begin{align}
		&\liminf_{\eps\downarrow 0}\Delta H^{m,n}_{\eps} \nonumber \\
		\ge~& 
		\min_{(b^1, b^2, b^3) \in B}
			\Big(-b^{1} \phi(\hat t,\hat x,\hat x') -b^{2}\big(\partial_{x}\phi(\hat t,\hat x,\hat x')+\partial_{x'}\phi(\hat t,\hat x,\hat x')\big) 
			- |b^{3}|^{2}\Delta \phi(\hat t,\hat x,\hat x') \Big) \nonumber
		\\ &-~ L  \varpi_{\circ}' \big( \rho_{\hat t} \big(\Pi^{n}_{t^{n}_{i}}(\xr) ,\Pi^{m}_{t^{n}_{i}}(\xr') \big) \big).
		\label{eq: borne inf Delta Hnm}
	\end{align}

	Similarly,  by \eqref{eq: linear growth unif n approximate sol} and \eqref{eq: hyp lip rho mu sig g},
	\begin{align*}
		& 
		-r^{n}_{i} \big(\hat  t,[\xr]^{n}_{i}\big) v^{n}_{i} \big( \hat t,   [\xr]^{n}_{i},\hat x  \big)
		+
		r^{m,n}_{i} \big( \hat t, [\xr']^{n}_{i}\big) v^{m,n}_{i} \big( \hat t, [\xr']^{n}_{i} ,\hat x'  \big)
		\\
		\ge&~
		-r^{n}_{i} \big( \hat  t, [\xr']^{n}_{i} \big) \phi  \big(\hat   t, \hat  x,\hat  y \big) 
		~-~ 
		\Delta r^{m,n}_{i} (\hat  t,[\xr]^{n}_{i},[\xr']^{n}_{i}) v^{m,n}_{i} \big( \hat t, [\xr']^{n}_{i} ,\hat x'  \big) ,
	 \end{align*} 
	where
	{
	$$
	\Delta r^{m,n}_{i}(\hat t,[\xr]^{n}_{i},[\xr']^{n}_{i})
		~:=~
		r^{m,n}_{i}(\hat t, [\xr']^{n}_{i})- r^{n}_{i}(\hat t,[\xr]^{n}_{i})
		$$
		satisfies}
	\begin{equation}\label{eq: borne r n eps i par varrhob} 
		\big| \Delta r^{m,n}_{i}(\hat t,[\xr]^{n}_{i},[\xr']^{n}_{i}) \big|
		~\le~
		L \rho_{\hat t}  \big( \Pi^{n}_{t^n_i} [\xr], \Pi^{m}_{t^n_i} [ \xr']  \big),
	\end{equation}
	Further, one has
	\begin{align*}
		&
		-\mu^{n}_{i} \big( \hat t, [\xr]^{n}_{i}\big)   \partial_{x}\phi \big(\hat t,\hat x,\hat x' \big)
		+
		\mu^{m,n}_{i} \big(\hat  t, [\xr']^{n}_{i}\big)(-\partial_{x'}\phi \big(\hat t,\hat x,\hat x' \big))
		\\
		=~&
		-\mu^{n}_{i}(\hat t,[\xr]^{n}_{i})\left[ \partial_{x}\phi \big(\hat t,\hat x,\hat x' \big)+ \partial_{x'}\phi \big(\hat t,\hat x,\hat x' \big)\right]   	-
		\Delta \mu^{m,n}_{i}(\hat t,[\xr]^{n}_{i},[\xr']^{n}_{i})   \partial_{x'}\phi \big(\hat t,\hat x,\hat x' \big),
	\end{align*} 
	where 
		$$
		\Delta \mu^{m, n}_{i}(\hat t,[\xr]^{n}_{i},[\xr']^{n}_{i})
		~:=~
		 \mu^{m,n}_{i}(\hat t, [\xr']^{n}_{i})-\mu^{n}_{i}(\hat t,[\xr]^{n}_{i}),
	$$
	satisfies
	\begin{equation}\label{eq: borne mu Yb} 
		\big| \Delta \mu^{m, n}_{i}(\hat t,[\xr]^{n}_{i},[\xr']^{n}_{i})) \big|
		~\le~
		L \rho_{\hat t}  \big( \Pi^{n}_{t^n_i} [\xr], \Pi^{m}_{t^n_i} [ \xr']  \big).
	\end{equation}

	Finally,  by \eqref{eq: inega IShii},
	\begin{align*}
		& 
		{\liminf_{\eps\downarrow 0}}\big(-\sigma^{n}_{i} \big(\hat  t, [\xr]^{n}_{i}\big) ^{2} \gamma_{\eps}
		+
		\sigma^{m,n}_{i} \big( \hat t, [\xr']^{n}_{i}\big))^{2}(-\gamma'_{\eps})\big)
		\\
		\ge ~&
		-  \sigma^{n}_{i} \big( \hat t, [\xr]^{n}_{i}\big)^{2} \partial^{2}_{xx} \phi \big(\hat t,\hat x,\hat x' \big)
		 - \big[\sigma^{n}_{i} \big( \hat t, [\xr]^{n}_{i}\big) +\Delta \sigma^{m, n}_{i}(t,[\xr]^{n}_{i},[\xr']^{n}_{i})  \big]^{2}\partial^{2}_{x'x'} \phi \big(\hat t,\hat x,\hat x' \big)
		\\
		&-2 \sigma^{n}_{i} \big( \hat t, [\xr]^{n}_{i}\big)\big[\sigma^{n}_{i} \big( \hat t, [\xr]^{n}_{i}\big) +\Delta \sigma^{m, n}_{i}(t,[\xr]^{n}_{i},[\xr']^{n}_{i})  \big] \partial^{2}_{xx'} \phi \big(\hat t,\hat x,\hat x' \big),
	 \end{align*} 
	 where 
	$$
		\Delta \sigma^{m, n}_{i}(\hat t,[\xr]^{n}_{i},[\xr']^{n}_{i})
		~:=~
		 \sigma^{m,n}_{i}(\hat t, [\xr']^{n}_{i})-\sigma^{n}_{i}(\hat t,[\xr]^{n}_{i}),
	$$
	so that
	\begin{equation}\label{eq: borne sig Yb} 
		\big| \Delta \sigma^{m, n}_{i}(\hat t,[\xr]^{n}_{i},[\xr']^{n}_{i}) \big|
		~\le~
		L \rho_{\hat t}  \big( \Pi^{n}_{t^n_i} [\xr], \Pi^{m}_{t^n_i} [ \xr']  \big).
	\end{equation}
	{By arbitrariness of $\eps>0$,} it follows that $w^{m,n}_i \big( \cdot,[\xr]^{n}_{i}, [\xr']^{n}_{i},\cdot, \cdot \big)$ is a viscosity subsolution on $[t^{n}_{i},t^{n}_{i+1})\x \R^{2}$ of 
	\begin{align*}
		0
		~\ge~ 
		-~ \partial_{t} \phi(  t,  x,  x') 
		~-~
		\max_{b\in B} \Lc^{b}_{i}[t,[\xr]^{n}_{i},[\xr']^{n}_{i}] \phi(t,x,y) - f^{n}_{i}(t,[\xr]^{n}_{i},[\xr']^{n}_{i},x),
	\end{align*} 
	in which
	\begin{align}\label{eq: def fni} 
		f^{n}_{i}(t,[\xr]^{n}_{i},[\xr']^{n}_{i},x)
		~:=~
		C~ \big(1+\|\xr'\|+|  x'| \big) ~ \varpi_{\circ}' \big( \rho_t  \big( \Pi^{n}_{t^n_i} [\xr], \Pi^{m}_{t^n_i} [ \xr']  \big) \big),
	\end{align}
	for some $C$ large enough, and 
	\begin{align*}
		&
		\Lc^{b}_{i}[t,[\xr]^{n}_{i},[\xr']^{n}_{i}]  \phi( t,  x,  x' ) \\
		:=&~
		(b^{1}+r^{n}_{i}) \big(    t, [\xr]^{n}_{i} \big) \phi  \big(   t,  x,  x' \big)
		+
		\big( b^{2}+\mu^{n}_{i}(t,[\xr]^{n}_{i}) \big) \big(\partial_{x}\phi(  t,  x,  x')+\partial_{x'}\phi(  t,  x,  x')\big)  
		\\
		&+\Delta \mu^{m,n}_{i}(t,[\xr]^{n}_{i},[\xr']^{n}_{i})   \partial_{x}\phi \big(  t,  x,  x' \big)
		+\frac12  |\sqrt{2}b^{3}|^{2}\Delta \phi(  t,  x,  x')  
		+
		\frac12  \sigma^{n}_{i} \big(   t, [\xr]^{n}_{i}\big)^{2}  \partial^{2}_{xx} \phi \big( t,  x,  x' \big)
		 \\
		 &+  \frac12 \big[\sigma^{n}_{i} \big(   t, [\xr]^{n}_{i}\big) +\Delta \sigma^{m, n}_{i}(t,[\xr]^{n}_{i},[\xr']^{n}_{i})  \big]^{2}  \partial^{2}_{x'x'} \phi \big( t,  x,  x' \big)
		\\
		&+ 
		\sigma^{n}_{i} \big(   t, [\xr]^{n}_{i}\big) \big[\sigma^{n}_{i} \big(   t, [\xr]^{n}_{i}\big) +\Delta \sigma^{m, n}_{i}(t,[\xr]^{n}_{i},[\xr']^{n}_{i})  \big]\partial^{2}_{xx'} \phi \big( t,  x,  x' \big).
	\end{align*}
 Moreover, it satisfies the boundary condition 
     \begin{align*} 
		&\lim_{t\uparrow t^{n}_{i+1}}  w^{m,n}_i \big( t,[\xr]^{n}_{i}, [\xr']^{n}_{i}, x, x' \big)
		~=~
		w^{m,n}_{i+1} \big(t^{n}_{i+1}, ([\xr]^{n}_{i}, x), ([\xr']^{n}_{i}, x'),  x, x' \big).
	\end{align*}

	\noindent $\mathrm{(ii)}$
	Given $(t,x,x')\in [t^{n}_{i},t^{n}_{i+1})\x \R^{2}$, $i<n$, and $\beta\in \Bc$, the collection of predictable $B$-valued processes, 
	on some probability space endowed with the augmented filtration of a two dimensional Brownian motion $W=(W^{1},W^{2})$, 
	let us now define the two processes $X^{\beta}$ and $X'^{\beta}$ by 
	\begin{align*}
		X^{\beta}=&
		\big [\xr \bp_{t}x \big]_{t \wedge} 
		+\!
		\int_{t}^{\cdot} \big[\beta^{2}_{s}+\mu \big(s, \Pi^{n}_{\eta_n(s)} [X^{\beta}] \big) \big] ds  
		+\!
		\int_{t}^{\cdot} \sqrt{2}\beta^{3}_{s} dW^{1}_{s}
		+\!
		\int_{t}^{\cdot} \sigma \big(s, \Pi^{n}_{\eta_n(s)} [X^{\beta}] \big) dW^{2}_{s},\\
		X'^{\beta}=& 
		\big[ \xr' \bp_{t}x' \big]_{t \wedge} 
		+
		\int_{t}^{\cdot }dX^{\beta}_{r}
		+
		\int_{t}^{\cdot} \Delta \mu \big( s, \Pi^{n}_{\eta_n(s)}[ X^{\beta}],\Pi^{m}_{\eta_m(s)} [ X'^{\beta}] \big) ds  \\
		&+~ \int_{t}^{\cdot}\Delta \sigma \big(s, \Pi^{n}_{\eta_n(s)} [ X^{\beta}],\Pi^{m}_{\eta_m(s)} [ X'^{\beta}] \big) dW^{2}_{s}.
	\end{align*}
	Assume that $\|\xr\|\vee \|\xr'\|\vee |x| \vee |x'|\le K$.
	We can find $C_{K}$, that only depends on $L$ and $K$, such that 
	$\E \big[ \big\|X^{\beta} \big\|^{2} + \big\| X'^{\beta} \big\|^{2} \big] \le C_{K}$,
	\begin{equation} \label{eq:X_PiX_diff}
		\E \Big[ 
			\big\|  \Pi^n \big[X^{\beta}\big]_{\eta^+_n(t) \vee} - X^{\beta}_{\eta^+_n(t) \vee} \big\|^2
		\Big]
		+
		\E \Big[ 
			\big\|  \Pi^m \big[X'^{\beta}\big]_{\eta^+_m(t) \vee} - X'^{\beta}_{\eta^+_m(t) \vee} \big\|^2
		\Big]
		\le
		C_{K} \big( |\pi^m|^{\frac12} + |\pi^n|^{\frac12} \big),
	\end{equation}
	and 	
	$$
		\E \Big[ 
			\big\|  \Pi^n \big[X^{\beta}\big]_{\eta^+_m(t) \wedge} - \Pi^n \big[ \xr \bp_{\eta^+_n(t)} x \big] \big\|^2
			+
			\big\|  \Pi^m \big[X'^{\beta}\big]_{\eta^+_m(t) \wedge} - \Pi^m \big[ \xr' \bp_{\eta^+_m(t)} x' \big] \big\|^2
		\Big]
		\le
		C_K |\pi^m|.
	$$
	Applying It\^o formula on $(X^{\beta}- X'^{\beta})^2$, together with \eqref{eq: hyp lip rho mu sig g}, \eqref{eq: borne mu Yb}, \eqref{eq: borne sig Yb} and \eqref{eq:X_PiX_diff}, 
	one can deduce that, for some constant $C_K' > 0$,
	\begin{align*}
		\E \Big[ \big \|X^{\beta}_{s\wedge} - X'^{\beta}_{s\wedge} \big\|^{2} \Big] 
		~\le~&
		|x-x'|^2 
		+
		C'_K \int_{\eta_m^+(t)}^s \E \Big[ \big \|X^{\beta}_{r \wedge} - X'^{\beta}_{r \wedge} \big\|^{2} \Big]  dr \\
		&+
		C'_K  \Big( \rho_{\eta^+_m(t) } \big(\Pi^n [ \xr \bp_{\eta^+_n(t) }x ], \Pi^m [ \xr' \bp_{\eta^+_m(t) }x'] \big)^2
			+ |\pi^m|^{\frac12} + |\pi^n|^{\frac12} \Big),
	\end{align*}
	and it follows by Gronwall's Lemma that, for some $C_K'' > 0$,
	\begin{align}\label{eq: ecart znm zn}
		\E \Big[ \big \|X^{\beta}_{\eta^+_m(t)\vee}-X'^{\beta}_{\eta^+_m(t)\vee} \big\|^{2} \Big]
		\le
		C_K'' \Big( \rho_{\eta^+_m(t) } \big(\Pi^n [ \xr \bp_{\eta^+_n(t) }x ], \Pi^m [ \xr' \bp_{\eta^+_m(t) }x'] \big)^2 + |\pi^m|^{\frac12} + |\pi^n|^{\frac12} \Big). 
	\end{align}
	By the Feynman-Kac formula, the  viscosity subsolution property of $w^{m,n}_i$, together with  \eqref{eq: def fni}, \eqref{eq: borne r n eps i par varrhob}, \eqref{eq: hyp borne rho} and the fact that $r$ is bounded by Assumption \ref{ass:FH}, 
	we deduce that 
	\begin{align*} 
		w^{m,n}_i \big( t,[\xr]^{n}_{i}, [\xr']^{n}_{i}, x, x' \big)
		~\le~
		\sup_{\beta\in \Bc}
		\E \Big[
			e^{\int_{t}^{T} k^{\beta}_{s}ds} ~L 
			\varpi_{\circ}' \big( \rho \big( \Pi^{n} [X^{\beta}], \Pi^{m} [X'^{\beta}] \big) \big)
			+
			\int_{t}^{T} e^{\int_{t}^{r} k^{\beta}_{s}ds} \zeta^{\beta}_{r}dr
		\Big],
	\end{align*}
	where, for all $\beta \in \Bc$, $k^{\beta}$ and $\zeta^{\beta}$ are predictable processes such that $|k^{\beta}|\le L$ and 
	$$
		|\zeta^{\beta}_r|
		~\le~
		L (1 + C_{K}) ~  
		\varpi_{\circ}' \big( \rho_r  \big(  \Pi^{n}_{\eta_n(r)} [  X^{\beta}], \Pi^{m}_{\eta_m(r)} [ X'^{\beta}] \big) \big)
		\big(1+\|X'^{\beta}\| \big) .
	$$
	Recalling that $\rho(\xr, \xr') \le C\|\xr- \xr'\|$, \eqref{eq: ecart znm zn} and that the function  $\varpi_{\circ}$ is concave,
	one can compute that
	\begin{align*}
		&~ \E \Big[ \varpi_{\circ}' \big( \rho_r  \big( \Pi^{n}[  X^{\beta}], \Pi^{m}[ X'^{\beta}] \big) \big)^2 \Big] \\
		\le~&
		C \varpi_{\circ} \Big( 
			\rho_{\eta^+_m(t) } \big(\Pi^n [ \xr \bp_{\eta^+_n(t) }x ], \Pi^m [ \xr' \bp_{\eta^+_m(t) }x'] \big)^2
			+
			\E \Big[ \big \| \Pi^{n}[  X^{\beta}]_{\eta^+_m(t) \vee} - \Pi^{m}[ X'^{\beta}]_{\eta^+_m(t) \vee} \big\|^{2} \Big]
		\Big) \\
		\le~&
		C_K \varpi_{\circ} \Big( 
			\rho_{\eta^+_m(t) } \big(\Pi^n [ \xr \bp_{\eta^+_n(t) }x ], \Pi^m [ \xr' \bp_{\eta^+_m(t) }x'] \big)^2 
			+
			|\pi^m|^{\frac12} 
			+ 
			|\pi^n|^{\frac12} 
		\Big).
	\end{align*}
	It follows that, for some constant $C_K > 0$, depending only on $K$,
	$$
		w^{m,n}_i \big( t,[\xr]^{n}_{i}, [\xr']^{n}_{i}, x, x' \big)
		\le
		C_K \varpi_{\circ}' \Big( 
			\rho_{\eta^+_m(t) } \big(\Pi^n [ \xr \bp_{\eta^+_n(t) }x ], \Pi^m [ \xr' \bp_{\eta^+_m(t) }x'] \big)
			+
			|\pi^m|^{\frac14}
			+ 
			|\pi^n|^{\frac14} 
		\Big).
	$$

	\noindent $\mathrm{(iii)}$ Finally, the corresponding lower-bound is derived similarly after exhibiting a viscosity supersolution property by the same arguments. 
	\endproof
	
	\begin{Remark} \label{rem:vn_Lipschitz}
	When $m = n$, one can bound $\E \big[\| \Pi^n[X^{\beta}]_{\eta^+_n(t) \vee} - \Pi^n[X'^{\beta}]_{\eta^+_n(t) \vee} \|^2\big]$ by
	$\E \big[ \|  X^{\beta}_{\eta^+_n(t) \vee} - X'^{\beta}_{\eta^+_n(t) \vee} \| \big]$ without using \eqref{eq:X_PiX_diff}.
	Letting $\varpi_{\circ}'(x) \equiv \varpi_{\circ}(x) \equiv x$ for all $x \ge 0$, the same argument will yield the estimation:
	for some constant $C_K > 0$,
	$$
		\big| v^n(t^n_i, \xr, \xr_{t^n_i}) - v^n(t^n_i, \xr', \xr'_{t^n_i}) \big|
		\le
		C_K \rho_{t^n_i} \big( \Pi^n[\xr], \Pi^n[\xr'] \big),
		~\mbox{when}~\| \xr\| + \|\xr'\| \le K.
	$$
	Moreover, by considering the PDE satisfied by $v^n$ on $[t^n_i, t^n_{i+1}] \x \R^d$, one can easily deduce that $v^n(t, \xr, x)$ is locally Lipschitz in $x$.
	More precisely, one can deduce that, for some constant $C_K > 0$,
	$$
		\big| v^n(t^n_i, \xr, \xr_{t^n_i}) - v^n(t^n_i, \xr', \xr'_{t^n_i}) \big|
		\le
		C_K \big\| \Pi^n_t[\xr]  - \Pi^n_t[\xr']  \big \|,
		~\mbox{when}~\| \xr\| + \|\xr'\| \le K.
	$$
	\end{Remark}
	
	We now discuss the stability with respect to the time argument. 
	\begin{Lemma} \label{lem:Unif_conti_vni_t}
		Let the conditions of Theorem \ref{thm: existence + Lipschitz cas differentiable}   hold 
		and  $v^n$ be the $\pi^{n}$-viscosity solution of \eqref{eq:PPDE} satisfying the terminal condition \eqref{eq: cond teminale vn} for each $n \ge 1$.
		Then, for all $K> 0$, there exists a constant $C_{K} > 0$,
		such that, for all $n \ge 1$, $0 \le t \le t' \le T$ and $\| \xr \| \le K$,
		\begin{equation} \label{eq:vn_holder_t}
		 \big| v^n(t',\xr_{t \wedge }, \xr_t) - v^n(t,\xr, \xr_t) \big|
			~\le~
			C_{K}  \varpi_{\circ}' \big( |t'-t|^{\frac{1}2} + |\pi^n|^{\frac14}\big). 
		\end{equation}
	\end{Lemma}
	\proof  	
	We use the notations introduced in the proof of Proposition \ref{prop: comparison n-visco}.    Fix $t'\le T$. Given $\theta:=(t,\xr, x) \in [0,t']\x \Omega \x \R^d$,  and $\mathfrak{a}\in \Ac$, let $(X^{\theta,\mathfrak{a}},Y^{\theta,\mathfrak{a}},Z^{\theta,\mathfrak{a}})$  be  the unique triplet of $\R^{d}\x \R\x \R^{d}$-adapted processes satisfying 
	$$
		\E \Big[ \|X^{\theta,\mathfrak{a}}\|^{2}+\|Y^{\theta,\mathfrak{a}}\|^{2}+\int_{t}^{t'}|Z^{\theta,\mathfrak{a}}_{s}|^{2}ds \Big]
		<
		\infty 
	$$
	and that solves the (decoupled) forward-backward SDE
	\begin{align*}
		X^{\theta,\mathfrak{a}}_{s}& 
		= 
		\1_{\{s \in [0,t)\}} \xr_s +
		\1_{\{s \in [t,T]\}} 
		\Big(
			x + \int_{t} ^{s} \underline\sigma \big(r,\Pi^{n}_{\eta_n(r)} [X^{\theta,\mathfrak{a}}], \mathfrak{a}_{r} \big) dW_{r}
		\Big), \\
		Y^{\theta,\mathfrak{a}}_{s}
		&=
		v^{n}(t',X^{\theta,\mathfrak{a}}\big) 
		+
		\int_{s}^{t'} \underline F \big( r,\Pi^{n}_{\eta_n(r)}[X^{\theta,\mathfrak{a}}] ,Y^{\theta,\mathfrak{a}}_{r},Z^{\theta,\mathfrak{a}}_{r} \big)dr
		-
		\int_{s}^{t'}Z^{\theta,\mathfrak{a}}_{r} dW_{r}, 
	\end{align*}
	for $t\le s\le t'$. Then,  Item (iv) of Assumption \ref{ass: F} implies that 
	(see the proof of Proposition \ref{prop: comparison n-visco})
	\begin{align}\label{eq: vn ge infa Ya} 
		v^{n}(t,\xr)
		~\ge~
		\inf_{\mathfrak{a}\in \Ac} Y^{\theta,\mathfrak{a}}_t.
	\end{align} 
 	By the Lipschitz property of $\underline F$, one can find a  predictable process $(k^{\mathfrak{a}},\zeta^{\mathfrak{a}})$ with values in $[-L,L]^{2}$ such that 
	$$
		\underline F \big( r,\Pi^{n}_{\eta_n(r)} [X^{\theta,\mathfrak{a}}] ,Y^{\theta,\mathfrak{a}}_{r},Z^{\theta,\mathfrak{a}}_{r} \big)
		~=~
		\underline F \big( r,\Pi^{n}_{\eta_n(r)}[X^{\theta,\mathfrak{a}}] ,0,0 \big)
		+
		k^{\mathfrak{a}}_{r} Y^{\theta,\mathfrak{a}}_{r} 
		+
		\zeta^{\mathfrak{a}}_{r} Z^{\theta,\mathfrak{a}}_{r},
		~~ t\le r\le t'.
	$$
	Let $\Q^{\mathfrak{a}}\sim \P$ be such that $W^{\mathfrak{a}}=W{-}\int_{0}^{\cdot}\zeta^{\mathfrak{a}}_{r}dr$ is a $\Q^{\mathfrak{a}}$-Brownian motion. Then, 
	\begin{align*}
		Y^{\theta,\mathfrak{a}}_{t} 
		&=
		\E^{\Q^{\mathfrak{a}}} \Big[
			e^{\int_{t}^{t'}k^{\mathfrak{a}}_{s}ds}v^{n} \big( t',X^{\theta,\mathfrak{a}}\big) 
			+
			\int_{t}^{t'}e^{\int_{t}^{r}k^{\mathfrak{a}}_{s}ds}\underline F\big( r,\Pi^{n}_{\eta_n(r)} [X^{\theta,\mathfrak{a}}] ,0,0 \big)dr 
		\Big].
	\end{align*}
 We can then find $C'_{K}>0$, that depends only on $L$ and $K$, such that, for all $\|\xr\|\le K$,  
	one has $\E^{\Q^{{\mathfrak a}}}[\|X^{\theta,\mathfrak{a}}_{(t' \vee t)\wedge}-\xr_{t}\|^{2}]^{\frac12}\le C'_{K} (t'-t)^{\frac12} $
	and 
	\begin{align*}
		Y^{\theta,\mathfrak{a}}_{t}-v^{n}(t',\xr_{ t\wedge})
		\ge \E^{\Q^{\mathfrak{a}}}\left[e^{\int_{t}^{t'}k^{\mathfrak{a}}_{s}ds}v^{n}(t',X^{\theta,\mathfrak{a}}\big)-v^{n}(t',\xr_{ t\wedge})-(t'-t)C'_{{K}}(1+\|\xr\|)\right].
	\end{align*}	
	Recall that the function $\varpi_{\circ}$ (used in the definition of $\varpi'_{\circ}$ in \eqref{eq:def_varpi_p}) is concave.
	One can then apply Lemma \ref{lem:Unif_conti_vni} and  \eqref{eq: linear growth unif n approximate sol} to deduce that there exists a constant $C''_{K}> 0$, 
	that depends only on $L$ and $K$, such that, for all $\|\xr \| \le K$,
 	$$
 		Y^{\theta,\mathfrak{a}}_{t}-v^{n}(t',\xr_{ t\wedge})
		~\ge~
		-C''_{K} \Big( 
			\big| t'-t \big| + e^{C''_{K}|t'-t|} - 1 
			+ 
			\varpi_{\circ}' \big( |t'-t|^{\frac{1}2} + |\pi^n|^{\frac14}\big)
		 \Big).
 	$$
	Recall that, for some constant $c> 0$, one has $\varpi_{\circ}(x) \ge cx$ for all $x \ge 0$.
	Then by \eqref{eq: vn ge infa Ya}, there exists a constant $C_K > 0$ such that
	$$
		v^{n}(t,\xr) - v^{n}(t',\xr_{ t\wedge})
		~\ge~
		- C_{K} \varpi_{\circ}' \big( |t'-t|^{\frac{1}2} + |\pi^n|^{\frac14}\big),
		~~\mbox{for all}~
		\|x \| \le K.
	$$
	A similar upper-bound is obtained by using the second inequality in  Item (iv)-(b) of Assumption \ref{ass: F}.
	\endproof

	\begin{Remark} \label{rem:vn_Lip_Hold}
	When $\varpi_{\circ} (x) := x$ for all $x \ge 0$, in view of Remark \ref{rem:vn_Lipschitz}, one can improve the above estimation in \eqref{eq:vn_holder_t} to be
	$$
		\big| v^n(t',\xr_{t \wedge }, \xr_t) - v^n(t,\xr, \xr_t) \big|
		~\le~
		C_{K}  |t'-t|^{\frac{1}2} . 
	$$
	In particular, in this case, $v^n$ is locally Lipschitz in $\Pi^n_t[\xr]$ and $1/2$-H\"older in $t$.
	\end{Remark}

We can now conclude the proof of   Theorem \ref{thm: existence + Lipschitz cas differentiable}.

	\proof[Proof of Theorem \ref{thm: existence + Lipschitz cas differentiable}]
	$\mathrm{(i)}$ Recall that $(\pi^{n})_{n \ge 1}$ is a sequence of discrete time grid on $[0,T]$, let  $v^n$ be a $\pi^{n}$-viscosity solution of \eqref{eq:PPDE} satisfying the terminal condition \eqref{eq: cond teminale vn}.
	Moreover, for each $(t, \xr) \in [0,T] \x \Om$, it follows by Lemma \ref{lem:Unif_conti_vni} that there exists a constant $C> 0$ such that, for all $n \ge m \ge 1$,
	\begin{equation} \label{eq:vmvn_diff}
		\big| v^m(t,\xr, \xr_t)- v^n(t,\xr, \xr_t) \big| 
		\le
		C \varpi_{\circ}' \Big( \rho_{\eta^+_m(t) } \big(\Pi^n [ \xr \bp_{\eta^+_n(t) } \xr_t ], \Pi^m [ \xr \bp_{\eta^+_m(t) } \xr_t] \big) 
			\!+\! 
			|\pi^m|^{\frac14}  
			\!+\!
			|\pi^n|^{\frac14} \Big).
	\end{equation}
	Therefore, $(v^n(t, \xr, \xr_t))_{n \ge 1}$ is a Cauchy sequence so that $v^n(t, \xr, \xr_t) \longrightarrow \vr(t, \xr)$ pointwise for some nonanticipative functional $\vr: [0,T] \x \Om \longrightarrow \R$.
	By Definition \ref{def: approximate solution}, $\vr$ is a $\pi$-approximate viscosity solution of \eqref{eq:PPDE}-\eqref{eq: cond bord}.
	Further, it is in fact the unique $\pi$-approximate viscosity solution, by the comparison principle in Proposition \ref{prop: comparison}.

	\vspace{0.5em}
	
	\noindent $\mathrm{(ii)}$
	By the estimation in \eqref{eq: regu en espace de vn} with $m=n$, together with the estimation in \eqref{eq:vn_holder_t}, 
	the limit function $\vr$ satisfies \eqref{eq: approximate sol is lipschitz cas diff}.
	
	\vspace{0.5em}

	\noindent $\mathrm{(iii)}$ We finally prove that the notion of approximate viscosity solution does not depend on the special choice of the increasing sequence of time grids. 
	Let us fix two increasing sequences of time grids $\pi=(\pi^{n})_{n\ge 1}$ and $\pi'=(\pi'^{n})_{n\ge 1}$, 
	and let $v^{n}$ be the $\pi^{n}$-viscosity solution, $v'^{n}$ be the $\pi'^{n}$-viscosity solution.
	We then define $\tilde \pi = (\tilde \pi^{n})_{ n\ge 1}$ by $\tilde \pi^k := \pi^k$ for $k \le n$, and $\tilde \pi^{k} := \pi^n \cup \pi'^{k-n}$ for $k > n$,
	and let $\tilde v^{k}$ be the corresponding $\tilde \pi^k$-viscosity solution.
	In particular, one has $\tilde \pi^{2n} := \pi^n \cup \pi'^{n}$,   $\tilde v^{n} = v^n$, and  that $\tilde v^{2n}$ is the $\pi^n \cup \pi'^{n}$-viscosity solution.
	Then, it follows from \eqref{eq:vmvn_diff} that, for some $C>0$, independent on $n$,
	\begin{align} \label{eq:diff_vn_v2n}
		\big| v^{n} - \tilde v^{2n} \big|(t,\xr,\xr_{t}) 
		&=
		\big| \tilde v^{n} - \tilde v^{2n} \big|(t,\xr,\xr_{t}) \nonumber \\
		&\le~
		C \varpi_{\circ}' \Big( \rho_{\eta^+_n(t) } \big(\Pi^n [ \xr \bp_{\eta^+_n(t) } \xr_t ], \widetilde \Pi^{2n} [ \xr \bp_{\tilde \eta^+_{2n}(t) } \xr_t] \big) 
			\!+\! 
			|\pi^n|^{\frac14}  
			\!+\!
			|\tilde \pi^{2n}|^{\frac14} \Big),
	\end{align}
	where $ \widetilde \Pi^{2n} $ and $\tilde \eta^+_{2n}$ are defined as $ \Pi^{2n}$ and $\eta^+_{2n}$ but with respect to $ \widetilde \pi^{2n}$.
	By the same arguments, one has
	$$
		\big| v'^{n} - \tilde v^{2n} \big|(t,\xr,\xr_{t}) 
		~\le~
		C \varpi_{\circ}' \Big( \rho_{\eta'^+_n(t) } \big(\Pi'^n [ \xr \bp_{\eta'^+_n(t) } \xr_t ], \widetilde \Pi^{2n} [ \xr \bp_{\tilde \eta^+_{2n}(t) } \xr_t] \big) 
			\!+\! 
			|\pi'^n|^{\frac14}  
			\!+\!
			|\tilde \pi^{2n}|^{\frac14} \Big),
	$$
	where $\Pi'^{n}$ and $\eta'^+_n$ are defined as $\Pi^{n}$ and $\eta^+_n$ but with respect to $\pi'^{n}$.
	Therefore, the sequences $(v^{n})_{n\ge 1}$ and  $(v'^{n})_{n\ge 1}$ have the same limit as $n \longrightarrow \infty$. 
	\endproof

\subsection{A classical solution is an approximate viscosity solution}
\label{subsec:Comparisons}
 	
	 We prove in this section that our notion of approximate solution is consistent with the notion of smooth solution (in the sense of Dupire). For this, we will have to restrict to the case where either $d\le 2$ or $\gamma\in \S^{d} \mapsto F(\cdot, \gamma)$ is  concave (or convex, by a change of variables). Note that Proposition \ref{prop: consistence solution smooth}  below combined with Remark \ref{rem: vitesse conv vn to v} actually provides a numerical algorithm for computing the smooth solution of a PPDE, if it exists. 	\vs2 
	
	Let us set  for this subsection 
	$$
		\Om = \DT
		~~\mbox{so that}~
		\Theta = [0,T] \x \DT,
	$$
	and recall the definition of derivatives of path-dependent functionals in the sense of Dupire \cite{dupireito}.
	Let $u: \Theta \longrightarrow \R$ be  a non-anticipative function,
	 and let us write 
	$$
		u(t, \xr, x) ~:=~ u(t, \xr \bp_t x) .
	$$
	The function $u: \Theta \longrightarrow \R$ is said to be horizontally differentiable if,
	for all $(t,\xr)\in[0, T) \x \DT$, the horizontal derivative
	$$
		\partial_{t}u(t,\xr) ~:=~ \lim_{h\searrow 0} \frac{u(t+h,\xr_{t \wedge }) - u(t,\xr_{t \wedge })}{h}\;\;
		~\mbox{is well-defined}.
	$$
	Next, $u$ is said to be vertically differentiable if, for all $(t,\xr)\in \Theta$, the function
	\begin{equation} \label{eq:u_xr_u_xr_y}
		y\in \R^{d} \longmapsto u(t, \xr, y) = u(t, \xr \bp_{t} y)
		~\mbox{is differentiable at}~\xr_t,
	\end{equation}
	whose derivative at $y= \xr_t$ is called the vertical derivative of $u$ at $(t, \xr)$, denoted by $\nabla_{\xr}u(t,\xr)$.
	Similarly, one can define the second-order vertical derivative $\nabla^{2}_{\xr}u (t, \xr)$.

	\vspace{0.5em}

	We say that $u$ is locally bounded if $\sup_{t \in [0,T], \|\xr\| \le K} |u(t, \xr)| < \infty$ for all $K> 0$.
	We denote by $\Cb_{b}^{\rm loc}(\Theta)$  the class of all non-anticipative, continuous, and locally bounded functions on $\Theta$,
	and {by 
	$\Cb^{1,2}(\Theta)$ the collection of elements $u \in \Cb_b^{\rm loc}(\Theta) $ such that $\partial_t u,$ $\nabla_{\xr}u$ and $\nabla^2_{\xr} u$ are all well defined, continuous   and locally bounded on $[0,T)\x D([0,T])$.}

	 We then say that $u\in \Cb^{1,2}(\Theta)$ is a classical solution of  \eqref{eq:PPDE}-\eqref{eq: cond bord} if 
	\begin{align*}
		- \partial_t u - F(\cdot, u, \nabla_{\xr} u, \nabla^2_{\xr} u) &= 0  \;\mbox{on $[0,T)\x D([0,T])$}\;\mbox{ and }\;
		u(T,\cdot)=g   \;\mbox{on $D([0,T])$.}
	 \end{align*}
	

	\begin{Proposition}\label{prop: consistence solution smooth} 
		Let Assumptions \ref{ass: F} and \ref{ass:FH} hold true,
		and ${\rm w}$ be a classical solution to the PPDE \eqref{eq:PPDE}-\eqref{eq: cond bord} such that 
		$$
			|{\rm w}(t,\xr)|\le C(1+\|\xr\|^{p}),\mbox{ for all }(t,\xr) \in \Theta,
		$$
		for some $C>0$ and $p\in \N$.  
		Assume in addition that either $d\le 2$ or $\gamma\in \S^{d} \mapsto F(\cdot, \gamma)$ is  concave.
		Then ${\rm w}$ is the approximate viscosity solution of \eqref{eq:PPDE}-\eqref{eq: cond bord}.
	\end{Proposition}

	 \begin{Remark} Note that the condition $d\le 2$ or $\gamma\in \S^{d} \mapsto F(\cdot, \gamma)$ is  concave is typically required for the existence of smooth solutions in the finite dimensional setting, see e.g.~\cite{lieberman1996second}. We shall actually use it in the proof below. Any other assumption ensuring the existence of smooth solutions in the finite dimensional setting (up to regularizing the coefficients) could be used in place of this one.  
	\end{Remark}
	
	\begin{proof}[Proof of Proposition \ref{prop: consistence solution smooth}] Let $v^n$ be the $\pi^n$-viscosity solution of \eqref{eq:PPDE}, 
	or, equivalently,  a viscosity solution of \eqref{eq: PDE sur ti ti+1}.
	It is enough to prove that $v^n \longrightarrow {\rm w}$ pointwisely as $n\to \infty$. 
	
	\vspace{0.5em}
	
	\noindent $\mathrm{(i)}$ In a first step, we notice that $F^n$ can be considered as a function of a $(\xr_{t^n_0}, \cdots, \xr_{t^n_n}, y, z, \gamma)$,
	and that one can regularize it  (and possibly add a small Laplacian term   to guarantee uniform ellipticity)  so that there is a sequence $(F^{n,k})_{k \ge 1}$ satisfying  
	$$
		F^{n,k}(\cdot, y, z, \gamma) = H^{n,k}( \cdot, y, z, \gamma) + r^{n,k} (\cdot) y + \mu^{n,k} (\cdot) \cdot z + \frac12 \Tr \big[ \sigma^{n,k} \sigma^{n,k \top} (\cdot) \gamma\big] , 
	$$
	for some smooth functions $(H^{n,k}, r^{n,k}, \mu^{n,k}, \sigma^{n,k})$ which converge to $(H^n, r^n, \mu^n, \sigma^n)$ uniformly.
	Moreover, when $d \le 2$ or $F$ is  concave in its last argument,
	the corresponding PDE \eqref{eq: PDE sur ti ti+1}, with generator $F^{n,k}$,  has a classical solution $v^{n,k} \in C^{1,3}_{b}$    which converges locally uniformly to $v^n$,  see e.g.~\cite[Theorem 14.15]{lieberman1996second} and \cite[Appendix D]{pham2014two}.
	In view of this, we can assume w.l.o.g. that $v^n$ is a smooth solution of  \eqref{eq: PDE sur ti ti+1} with generator $F^n$.

	\vspace{0.5em}
	
	\noindent $\mathrm{(ii)}$ 		
	Next, we will consider $v^n$ as a path-dependent functional  and  essentially argue as in  the proof of  Lemma \ref{lem:Unif_conti_vni} with the additional property (recalling the convention \eqref{eq:u_xr_u_xr_y}) that 
	$$
		\phi^n(t, \xr, \xr', \xr_t, \xr'_t) = \phi^n(t, \xr,\xr') := {\rm w}(t, \xr) - v^n(t, \xr')
	$$
	is a smooth functional.  We therefore only sketch the proof. 
	For later use, note that 
	$$
		 \nabla_{\xr} \phi^n(t, \xr, \xr') = \nabla_{\xr} {\rm w}(t, \xr  )  
	\;\mbox{ , }\;
		\nabla_{\xr'} \phi^n(t, \xr, \xr') = - \nabla_{\xr'} v^{n}(t,  \xr')\;\mbox{ and }\; \nabla_{\xr'} \nabla_{\xr}\phi^n(t, \xr, \xr')=0.
	$$
	 For ease of notations, we pursue the proof with $d=1$. We first compute that
	\begin{align*}
		&
		-H \big(t, \xr , ({\rm w}, \nabla_{\xr} {\rm w}, \nabla^2_{\xr} {\rm w})(t,\xr   \big) 
		+ 
		H^n \big(t, \xr', (v^n, \nabla_{\xr} v^n, \nabla^2_{\xr} v^n)(t, \xr') \big)  \\
		& =
		- f(t,\xr,\Pi^{n}_{\eta_n(t)} (\xr')) - \kappa^{n}(t,\xr,\xr')  \phi^n(t, \xr,\xr')- b^{n}(t,\xr,\xr')(\nabla_{\xr} \phi^n(t, \xr,\xr')+\nabla_{\xr'} \phi^n(t, \xr,\xr'))\\
		&\;\;- a^{n}(t,\xr,\xr')\big(\nabla^{2}_{\xr} \phi^n(t, \xr,\xr')+\nabla^{2}_{\xr'} \phi^n(t, \xr,\xr')\big)
		\end{align*}
	for some  continuous map $a^{n}$ that is non-negative and bounded by $L$,  a continuous map $(\kappa^{n},b^{n})$ bounded by $2L$, and where 
	$f$ is continuous and satisfies 
	$$
		\big| f(t,\xr,\Pi^{n}_{\eta_n(t)}(\xr'))\big|
		~\le~
		L\rho_{t}(\xr,\Pi^{n}_{\eta_n(t)}(\xr')),
	$$ 
	recall Assumption \ref{ass:FH}. Note that $a^n,b^n,\kappa^n$ depend on $(t,\xr)$ through ${\rm w}$ and $v^n$ possibly.  Let us set 
	$$
		r^{n}(t,\xr) 
		~:=~
		r \big(t, \Pi^n_{\eta_n(t)}[\xr] \big),
	$$
	and define $\mu^{n}$ and $\sigma^{n}$ similarly with respect to $\mu$ and $\sigma$. 
	 	Next, we compute that	
	\begin{align*}
		&-\Big( r(t, \xr   ) {\rm w}(t, \xr )  + \mu(t, \xr  )  \nabla_{\xr} {\rm w}(t, \xr ) + \frac12  \sigma (t, \xr)^{2} \nabla^2_{\xr}{\rm w}(t, \xr ) \Big) \\
		&+
		\Big( r^n(t, \xr') v^n(t, \xr')  + \mu^n(t, \xr')  \nabla_{\xr'} v^n(t, \xr') + \frac12  \sigma^n  (t, \xr')^{2} \nabla^2_{\xr'} v^n(t, \xr')  \Big) \\
		=&~
		-r^n(t, \xr') \phi^n(t, \xr, \xr') - \Delta r^n(t, \xr, \xr')  {\rm w}(t, \xr ) \\
		&
		-\mu^n(t, \xr') \big( \nabla_{\xr} \phi^n(t, \xr, \xr')+\nabla_{\xr'} \phi^n(t, \xr, \xr')\big) - \Delta \mu^n(t, \xr, \xr')  \nabla_{\xr} \phi^{n}(t, \xr,\xr') \\
		&-
		\frac12  \sigma^n (t, \xr')^2 \nabla^2_{\xr'} \phi^n(t, \xr, \xr') 
		\\&-
		\frac12 (\sigma^n+ \Delta \sigma^n)  (t, \xr, \xr')^2 \nabla^2_{\xr} \phi^n(t, \xr, \xr'),
	\end{align*}
	where $ \Delta r^n(t, \xr, \xr')  :=   r(t, \xr ) - r^n(t, \xr') $ satisfies
	$$
		\big| \Delta r^n(t, \xr, \xr') \big|  ~\le~L \rho_t \big( \xr , \Pi^n_{\eta_n(t)}[\xr'] \big),
	$$
	and $\Delta \mu^n(t, \xr, \xr')  := \mu(t, \xr ) - \mu^n(t, \xr') $, $\Delta \sigma^n(t, \xr, \xr')  := \sigma(t, \xr)  - \sigma^n(t, \xr') $ satisfy also
	$$
		\big| \Delta \mu^n(t, \xr, \xr') \big| + \big| \Delta \sigma^n(t, \xr, \xr') \big|
		~\le~
		 2 L  \rho_t \big( \xr , \Pi^n_{\eta_n(t)}[\xr'] \big).
	$$
	Therefore, $(t, \xr, \xr') \mapsto \phi^n(t, \xr, \xr')$  satisfies on $[0,T)\x D([0,T])^{2}$ 
	\begin{align*}
		0 = & -(\kappa^{n}(t,\xr,\xr')+r^n(t, \xr')) \phi^n(t, \xr,\xr')
		- \partial_t \phi^n (t,\xr,\xr')\\
		&- (b^{n}(t,\xr,\xr')+\mu^n(t, \xr'))(\nabla_{\xr} \phi^n(t, \xr,\xr')+\nabla_{\xr'} \phi^n(t, \xr,\xr')) - \Delta \mu^n(t, \xr, \xr')  \nabla_{\xr} \phi^{n}(t, \xr,\xr') \\
		&\;\;- a^{n}(t,\xr,\xr')\big(\nabla^{2}_{\xr} \phi^n(t, \xr,\xr')+\nabla^{2}_{\xr'} \phi^n(t, \xr,\xr')\big)\\
		&-
		\frac12  \sigma^n (t, \xr')^2 \nabla^2_{\xr'} \phi^n(t, \xr, \xr') 
		-\frac12 (\sigma^n+ \Delta \sigma^n)  (t, \xr, \xr')^2 \nabla^2_{\xr} \phi^n(t, \xr, \xr')\\
		&-( \Delta r^n(t, \xr, \xr')  {\rm w}(t, \xr ) +  f(t,\xr,\Pi^{n}_{\eta_n(t)}(\xr')))
	\end{align*}
	with 
	$$
	\phi^n (T,\xr,\xr')=g(\xr)-g(\Pi^{n}(\xr')), \mbox{ for } \xr,\xr'\in \Omega.
	$$
	Since $\nabla_{\xr}\nabla_{\xr'}\phi^{n}(\cdot,\xr,\xr')=0$,
	one can deduce that 
	$$
		\lim_{n \longrightarrow \infty} \phi^{n}(t,\xr,\xr)= 0,
		~\mbox{for all}~
		(t,\xr)\in \Theta,
	$$
	by the same arguments as in  (ii) of  the proof of Lemma  \ref{lem:Unif_conti_vni} by using \cite[Theorem 3.16]{possamai2020zero},  the functional It\^{o}'s formula of \cite{cont2013functional} and the corresponding Feynman-Kac's formula (recall that $v^{n}$ and ${\rm w}$ have polynomial growth).
	\end{proof}

\section{Regularity of approximate viscosity solutions}
\label{sec:Regularity}

	In this section, we   discuss the regularity of $\pi$-approximate viscosity solutions for a fully nonlinear PPDE and a semi-linear PPDE.
	Namely, we provide in this section some conditions under which the Dupire's vertical derivative exists and is continuous.
	One of the motivations to study such a regularity property is the fact that it allows to apply the functional It\^{o} formula in \cite{bouchard2021c}. 
	The study of further regularity of the solution, in particular the potential regularization effect of a uniform ellipticity condition,   is left for future researches. 

	\vspace{0.5em}

	All over this section, we set
	$$
		\Om = \DT
		~~\mbox{so that}~
		\Theta = [0,T] \x \DT.
	$$
	 We also recall that, for a functional defined on $\Theta$, its derivatives in the sense of Dupire is recalled in Section \ref{subsec:Comparisons}. 

\subsection{A fully nonlinear case}

A particular fully nonlinear operator has already been studied in \cite[Section 3.2]{bouchard2021c} in the case $d=1$ and for a PPDE of the form 
$
-\partial_{t}v-H(\nabla^{2}_{\xr} v)=0
$
with a convex function $H$. 
We provide here a  generalization which is proved by an appropriate modification of the arguments contained in the proof of  \cite[Proposition 3.13]{bouchard2021c} and by appealing to the a-priori estimate of  Lemma \ref{lem:Unif_conti_vni_t} and Remark \ref{rem:vn_Lip_Hold}. 
	In particular, we  allow for more general nonlinearities when the terminal condition is $C^{1+1}$ (in the sense of Assumption \ref{ass: regul fully non linear} below when $\alpha=1$). It fully uses the piecewise constant approximation of the path, embedded in the definition of approximate viscosity solutions, to reduce the analysis to the finite dimensional case which provides uniform estimates   on the corresponding sequence of $\pi^{n}$-viscosity solutions, for all $n\ge 1$.  

	\begin{Assumption} \label{ass: regul fully non linear}
		$\mathrm{(i)}$ The function $g\in \Cc_{L}$ and there exists  $\alpha \in (0,1]$ and  a finite positive measure $\lambda$ on $[0,T]$  such that,
		for all $\xr,\xr'\in \DT$,   $B = [s, t) \subset [0,T]$ and $\delta\in {\R^{d}}$,
		\begin{equation}\label{eq: g Lispch}
			|g(\xr)-g(\xr')|\le  \int_{0}^{T}|\xr_{s}-\xr_{s}'|\lambda(ds),
		\end{equation} 
		$\delta'\in {\R^{d} }\longmapsto g(\xr+\delta'\1_{B})$   is differentiable and\footnote{We abuse notations to write $d g(\cdot+\delta\1_{B}+\cdot)/d\delta$ for the corresponding derivative with respect to $\delta$.} 
		\begin{equation}\label{eq: measure sur derive}
			\Big| \frac{d g(\xr+\delta\1_{B}+\xr')}{d\delta} - \frac{ d  g(\xr+\delta\1_{B})}{d \delta} \Big |
			\le 
			\Big( \int_{0}^{T} \big| \xr'_{s} \big| \lambda(ds) \Big)^{\alpha} \lambda(B).
		\end{equation}
 		
		\noindent $\mathrm{(ii)}$ {For} any increasing sequence $0=t_{0}<t_{1}<\cdots<t_{n}=T$ with $\max_{i<n}|t_{i+1}-t_{i}|$ small enough, 
		for all $1\le i<j< n$ there exists  a diagonal matrix $p^{i,j}\in \S^{d}$ with non-zero diagonal terms  such that, 
		for all   $\delta \in \R^{d}$, and $(x_{\ell})_{0\le \ell {\le n}}\subset (\R^{d})^{{n+1}}$,
		\begin{align} \label{eq: hyp pi}
			&g\left( \sum_{\ell=0}^{n-1} (x_{\ell}+ { \delta} \1_{\{\ell = i\}})\1_{[t_{\ell},t_{\ell+1})} + \1_{\{T\}} { x_{{n}}} \right) \nonumber \\
			&= g\left(\sum_{\ell=0}^{n-1} (x_{\ell}+p^{i,j} {\delta}\1_{\{\ell\ge j\}})\1_{[t_{\ell},t_{\ell+1})} + \1_{\{T\}} ({ x_{{n}}} +p^{i,j} {\delta}) \right).
		\end{align}
		
		\noindent $\mathrm{(iii)}$ $F$ satisfies the conditions of Theorem \ref{thm: existence + Lipschitz cas differentiable}.
		 Moreover, $F$ is concave\footnote{or convex, which is equivalent up to replacing $F(\xr,\cdot)$ by $-F(\xr,-\cdot)$.} in $\gamma$ or $d \le 2$.
		Further, one of the following holds: 
			\begin{itemize}
			\item[{\rm (a)}]
			Either $\alpha\in (0,1)$ and 
			\begin{align}\label{eq: forme particuliere F regu fully non linear alpha le 1} 
			F(t,\xr,y,z,\gamma)=  F_{1}(t)y+F_{2}(t)  z+F_{3}(t,\gamma),\; (t,\xr,y,z,\gamma)\in \Theta\x \R\x \R^{d}\x \S^{d},
			\end{align}
			for some continuous maps $F_{i}$, $i=1,\ldots,3$.
			\item[{\rm (b)}] Or $\alpha=1$ and 
			\begin{align}\label{eq: F cas alpha=1} 
			F(t,\xr,y,z,\gamma)=F_{1}(t,y,\gamma)+F_{2}(t)z
			\end{align}
			for some continuous maps $F_{i}$, $i=1,2$, such that 
			  $y\in \R \longmapsto F_{1}(t,y,\gamma)$ is $C^{1}$ with bounded and Lipschitz first order derivative, uniformly in $\gamma\in \R^{d}$ and $t\le T$.  			
			  \end{itemize}
	\end{Assumption}
	
	\begin{Example}
		Let  {$d=1$,} $g_0:  {\R} \longrightarrow \R$ be in $C^{1+\alpha}$, and $\lambda$ be a finite positive measure on $[0,T]$ such that $\lambda([T-\eps, T]) > 0$ for any $\eps > 0$.
		Then the following $g: \DT \longrightarrow \R$ {defined by}
		$$
			g(\xr) ~:=~ g_0 \Big( \int_0^T \xr_t \lambda(dt) \Big)
		$$
		satisfies the conditions of Assumption \ref{ass: regul fully non linear}.
	\end{Example}

	Let us comment on the conditions in Assumption \ref{ass: regul fully non linear}.
	In \cite{bouchard2021c}, we essentially ``differentiate'' twice the $\pi^{n}$-viscosity solution $v^n$ of the regularized version of the PDE,
	and then derive the uniform boundedness and H\"older type estimates on the corresponding gradients of $v^n$.
	For this, we interpret the differentiated PDEs in terms of flow equations to propagate backward the regularity of the terminal condition $g$, see \eqref{eq: g Lispch}-\eqref{eq: measure sur derive}.
	The main difficulty is that a bump at time $t\in [t^{n}_{i},t^{n}_{i+1})$ on  $x$ induces a bump on the $i+1$-th element of $[\xr\boxplus_{t^{n}_{i+1}} x ]^{n}_{i+1}$ on $[t^{n}_{i+1},t^{n}_{i+2})$,
	because of the boundary condition 
	$$
		\lim_{t \nearrow t^{n}_{i+1},x'\to x} 
		v^{n}(t,\xr, x')= v^n \big( t^{n}_{i+1}, \Pi^n[\xr \bp_{t^{n}_{i+1}} x], x \big).
	$$
	Therefore, it turns, from the next term interval on, into a bump on a parameter of the PDEs. See also \eqref{eq: cond bord phi n j} in the proof below.  
	The structure condition \eqref{eq: hyp pi}   allows one  to transform a derivative with respect to a past value of the path into a derivative with respect to the end of the path. More precisely, we have the following. 

	\begin{Remark}\label{rem : transfert derives} 
		Assume that $v^{n}$ is a smooth $\pi^{n}$-viscosity solution of \eqref{eq:PPDE}-\eqref{eq: cond bord} and that   Assumption   \ref{ass: regul fully non linear} holds. 
		Then, a backward induction argument based on {\rm (ii)} of Assumption   \ref{ass: regul fully non linear} and the fact that $F$ does not depend on the path $\xr$ imply that 
$v^{n}(t,\cdot+\delta \1_{[t_{\ell},t_{\ell+1})} ,\cdot)=v^{n}(t,\cdot,\cdot+p^{\ell,i+1}\delta)$ for $t\in [t^{n}_{i},t^{n}_{i+1})$, $\ell <i+1\le n$, for some $p^{\ell,i+1}$. Letting $v^{n}_{i}$ be defined as in \eqref{eq: def vni} and denoting by $\partial_{\ell} v^{n}_{i}(t,[\xr]^{n}_{i},x)$ the gradient of $v^{n}_{i}(t,[\xr]^{n}_{i},x)$ with respect to the $\ell$-th component of $[\xr]^{n}_{i}$, we deduce that
		$$
			\partial_{\ell} v^{n}_{i}(t,[\xr]^{n}_{i},x)=p^{\ell,i+1}D v^{n}_{i}(t,[\xr]^{n}_{i},x).
		$$ 
		Note that this is no more possible if $F$ depends on $\xr$.
	\end{Remark}

	As for \eqref{eq: g Lispch}-\eqref{eq: measure sur derive}, it is used to absorb the cumulation of derivatives on different time steps generated by the above mentioned phenomena, see \eqref{eq: borne partial gn}-\eqref{eq: borne partial2 gn} below. 
	Notice that \eqref{eq: g Lispch}-\eqref{eq: measure sur derive} is a $C^{1+\alpha}$-type regularity assumption on $g$. It actually turns into a $C^{1+\alpha}$-regularity in space on the solution $\vr$ of the PPDE.  
	This translates, as usual, into an H\"older continuity in time, up to an additional term associated to the measure $\lambda$. 
	The interior regularity that could be obtained under a uniform ellipticity condition is by far not obvious and remains an open question.

	\begin{Theorem}\label{thm: C1alpha avec pji}  	
		Let Assumption \ref{ass: regul fully non linear}  hold true. Let $\vr$ be the unique $\pi$-approximate viscosity solution of \eqref{eq:PPDE}-\eqref{eq: cond bord}.
		Then, the vertical derivative $\nabla_{\xr} \vr(t,\xr)$ is well defined for all $(t, \xr) \in [0,T]\x\DT$, and there exists $C>0$ such that 
		$|\nabla_{\xr} \vr(t,\xr)|\le C$,
		\begin{equation} \label{eq:DV_vertical_continuity}
			|\nabla_{\xr} \vr(t,\xr')- \nabla_{\xr} \vr(t,\xr)|
			~\le~
			C \Big( \Big|\int_{0}^{t} |\xr'_{s}-\xr_{s}| \lambda(ds) \Big|^{\alpha}+|\xr'_{t}-\xr_{t}|^{\alpha}\Big),
		\end{equation}
		for all $(t, \xr, \xr') \in [0,T] \x \DT \x \DT$.
		Moreover, for all $K > 0$, there exists a constant $C_K> 0$ such that
		\begin{align}\label{eq: unif conti en temps de nabla v}
			|\nabla_{\xr} \vr(t',\xr_{t \wedge })- \nabla_{\xr} \vr(t,\xr)|
			~\le~
			C_K \Big( |t'-t|^{\frac{\alpha}{2+2\alpha}}+\limsup_{\eps\downarrow 0}\lambda([t-\eps,t'+\eps)) \Big),
		\end{align}
		for all $t \le t' \le T$ and $\xr\in \DT$ satisfying $\|\xr\| \le K$.
	\end{Theorem}

\begin{proof} 
	For ease of notations, we restrict to $d=1$. When $d>1$, it suffices to derive the following estimates for each dimension separately.
	By (i) of Assumption \ref{ass: regul fully non linear}, upon smoothing
	$$
		g^{n}: (x_{0},\ldots,x_{n})\in \R^{n+1} \longmapsto g \Big( \sum_{i=0}^{n-1} x_{i}\1_{[t^{n}_{i},t^{n}_{i+1})}+x_{n}\1_{\{T\}} \Big),
	$$
	we can assume that it is in $C^{\infty}$ and that 
	\begin{align}
		|\partial_{x_{i}} g^{n} (x_{0},\ldots,x_{n})|&\le \lambda([t^{n}_{i},t^{n}_{i+1})),\;i,j\le n, \;\mbox{if $\alpha\in (0,1]$,}\label{eq: borne partial gn} \\
		|\partial_{x_{i}x_{j}}^{2} g^{n} (x_{0},\ldots,x_{n})|&\le \lambda([t^{n}_{i},t^{n}_{i+1}))\lambda([t^{n}_{j},t^{n}_{j+1})),\;i,j\le n, \;\mbox{if $\alpha=1$,}\label{eq: borne partial2 gn} 
	\end{align}
	with the convention $\lambda([t^{n}_{n},t^{n}_{n+1}) ) :=\lambda(\{t^{n}_{n}\})$.
	Let $V^{n}$ be defined by
	$$
		V^{n}(t,[\xr]^{n}_{i},x)=v^{n}(t,\xr,x),\;(\xr,x)\in D([0,T])\x \R^{d},\;t\in [t^{n}_{i},t^{n}_{i+1}), \;i<n.
	$$

	\noindent $\mathrm{(i)}$ 
	Upon  smoothing also $F$ and adding a uniform elliptic term, we can   assume    that  $V^{n}$ is $C^{1,3}_{b}$ on $[0,T)${, see e.g.~\cite[Theorem 14.15]{lieberman1996second} and \cite[Appendix D]{pham2014two}}. The estimates below will not depend upon this smoothing procedure.   Let $\phi^{n,j}$ denote the gradient of $V^{n}$ with respect to the $j$-th space component, and denote by $D\phi^{n,j}$ and  $D^{2}\phi^{n,j}$ the first and second order derivatives of $\phi^{n,j}$ with respect to its last space variable. On each $[t^{n}_{i},t^{n}_{i+1})$,  $\phi^{n,j}$, $j\le i+1$,   solves 
$$
0=-\partial_{t}\phi^{n,j} -\partial_{y }F(\Xi)\phi^{n,j}-\partial_{z}F (\Xi) D\phi^{n,j}-\partial_{\gamma}F(\Xi)D^{2}\phi^{n,j} 
$$
where $\Xi:=(\cdot,V^{n},DV^{n},D^{2}V^{n})$, 
with the boundary condition 
\begin{align}\label{eq: cond bord phi n j} 
\lim_{t\uparrow t^{n}_{i+1}} \phi^{n,j}(t,y,x)=\phi^{n,j}(t^{n}_{i+1},(y,x),x)+\1_{\{j=i+1\}}\phi^{n,i+2}(t^{n}_{i+1},(y,x),x),\;(y,x)\in \R^{i}\x \R. 
\end{align}
By Feynman-Kac's formula and Assumption \ref{ass: regul fully non linear},  for each $(t,z)\in [t^{n}_{i},t^{n}_{i+1})\x \R^{i+1}$, $j\le i+1\le n$, we can then find a process $X$ and a bounded random variable $\beta$, with bound $C_{\beta}>0$ depending only on $F$, such that 
	$$
		\phi^{n,j}(t,z)
		=
		\E \Big[ \beta \Big( 
			\partial_{j}g^{n}(X_{t^{n}_{0}},\ldots,X_{t^{n}_{n}})+\1_{\{j=i+1\}}\sum_{j'=i+2}^{n} \partial_{j'}g^{n}(X_{t^{n}_{0}},\ldots,X_{t^{n}_{n}}) 
		\Big) \Big].
	$$
By \eqref{eq: borne partial gn}, it follows that  
\begin{align}\label{eq: borne phi n j} 
|\phi^{n,j}(t,\cdot)|\le C_{\beta}\left(\lambda([t^{n}_{j},t^{n}_{j+1}))+\1_{\{j=i+1\}}\lambda([t^{n}_{i+2},T])\right), \; \mbox{on } [t^{n}_{i},t^{n}_{i+1}),\; j\le  i+1\le n. 
\end{align}

	\noindent $\mathrm{(ii)}$
	In the case $\alpha<1$, the rest of the proof is exactly the same as the proof of  \cite[Proposition 3.13]{bouchard2021c} up to some obvious modifications. 
	In fact, it suffices to replace their two first estimates at the very beginning of part 4.c.~of their proof by 
	\eqref{eq: borne phi n j} and \eqref{eq:vn_holder_t}  (together with Remark \ref{rem:vn_Lip_Hold}). 

	\vspace{0.5em}

	\noindent $\mathrm{(iii)}$ 
	In the following, we consider the case $\alpha=1$, where the argument is slightly different and we only sketch the proof.

	\vspace{0.5em}

	(a) Let us set $\psi^{n,j,\ell}:=\partial_{x_{\ell}} \phi^{n,j}$, which solves the PDE
\begin{align}
0=& -\partial_{t}\psi^{n,j,\ell} -\partial_{y }F(\Xi)\psi^{n,j,\ell}-\partial_{z}F (\Xi) D\psi^{n,j,\ell}-\partial_{\gamma}F(\Xi)D^{2}\psi^{n,j,\ell}\nonumber
\\
&- \phi^{n,j}\left( \partial^{2}_{yy }F(\Xi) \phi^{n,\ell}+\partial^{2}_{yz }F(\Xi)D\phi^{n,\ell}+\partial^{2}_{y\gamma }F(\Xi) D^{2}\phi^{n,\ell}\right)\label{eq: dyz}\\
&-D\phi^{n,j}\left( \partial^{2}_{zy }F(\Xi) \phi^{n,\ell}+\partial^{2}_{zz }F(\Xi)D\phi^{n,\ell}+\partial^{2}_{z\gamma }F(\Xi) D^{2}\phi^{n,\ell}\right)\nonumber\\
&-D^{2}\phi^{n,j}\left( \partial^{2}_{\gamma y }F(\Xi) \phi^{n,\ell}+\partial^{2}_{\gamma z }F(\Xi)D\phi^{n,\ell}+\partial^{2}_{\gamma\gamma }F(\Xi) D^{2}\phi^{n,\ell}\right).\nonumber
 \end{align}
 By \eqref{eq: F cas alpha=1}  and Remark \ref{rem : transfert derives}, we can then find $q^{\ell,i+1}\in \R$, that depends on $n$, such that
 \begin{align*}
0=& -\partial_{t}\psi^{n,j,\ell} - \partial_{y }F(\Xi)  \psi^{n,j,\ell}\\
&-\left(\partial_{z}F (\Xi)+q^{\ell,i+1}\phi^{n,j}\partial^{2}_{y\gamma }F(\Xi)+q^{\ell,i+1}D\phi^{n,j} \partial^{2}_{z\gamma }F(\Xi)\right) D\psi^{n,j,\ell}\\
&-q^{\ell,i+1}\left( \partial^{2}_{\gamma y }F(\Xi) \phi^{n,\ell}+\partial^{2}_{\gamma z }F(\Xi)D\phi^{n,\ell}+\partial^{2}_{\gamma\gamma }F(\Xi) D^{2}\phi^{n,\ell}\right)D\psi^{n,j,\ell}\\
&-\partial_{\gamma}F(\Xi)D^{2}\psi^{n,j,\ell}- \phi^{n,j} \partial^{2}_{yy }F(\Xi) \phi^{n,\ell} .
 \end{align*}
Moreover, 
\begin{align*}
\lim_{t\uparrow t^{n}_{i+1}}\psi^{n,j,\ell}(t,y,x)=&	\1_{\{\ell < i+1\}} \left( \psi^{n,j,  {\ell}} +\1_{\{j=i+1\}}\psi^{n,i+2, {\ell}} \right) (t^{n}_{i+1},(y,x),x)\\
		&
		+\1_{\{\ell=i+1\}} 
		\left( \psi^{n,j, {i+2} }  +\1_{\{j=i+1\}}\psi^{n,i+2,  {i+2} }\right) (t^{n}_{i+1},(y,x),x),
\end{align*}
for $(y,x)\in \R^{i}\x \R$. Then, by (iii) of Assumption \ref{ass: regul fully non linear} and  \eqref{eq: borne phi n j}, one can apply the Feynman-Kac's formula again to find, for given $(t,z)\in [t^{n}_{i},t^{n}_{i+1})\x \R^{i+1}$, $i< n$,  and $j,\ell\le i+1$,  adapted processes $X, \beta^{1}$ and $\beta^{2}$, that depend on $n$ but such that $|(\beta^{1},\beta^{2})|\le \tilde C_{\beta}$ for some $\tilde C_{\beta}>0$ that only depend on $F$ and $g$, for which 
	\begin{align*}
		\psi^{n,j,\ell}(t,z)
		=&
		\1_{\{\ell < i+1\}} 
		\E \Big[\beta^{1}_{T} \Big(\partial^{2}_{x_{j}x_{\ell}} g^{n} 
			+
			\1_{\{j=i+1\}} \!\!\! \sum_{j'=i+2}^{n}\partial^{2}_{x_{j'}x_{\ell}} g^{n}  \Big) 
			\big( \Pi^{n}[X] \big) 
		\Big]\\
		&+\1_{\{\ell=i+1\}}  \sum_{k=\ell+2}^{n}
		\E \Big[\beta^{1}_{T}  \Big(\partial^{2}_{x_{j}x_{k}} g^{n} 
			+
			\1_{\{j=i+1\}} \!\!\! \sum_{j'=i+2}^{n-1} \partial^{2}_{x_{j'}x_{k}}g^{n} \Big) 
			\big( \Pi^{n}[X] \big) 
		\Big]\\
		&+ \E \Big[\int_{t}^{T} \beta^{2}_{s}ds\Big].
	 \end{align*}
	Thus, by \eqref{eq: borne partial2 gn}, 
	\begin{align}
		|\psi^{n,j,\ell}(t,z)|&\le \tilde C_{\beta} \lambda([0,T])^{2}+\tilde C_{\beta}T.\label{eq: borne psi n j ell}
	\end{align}

	(b)  Recall that $\alpha =1$.
	It follows from \eqref{eq:vn_holder_t}  (by setting $\varpi_{\circ}(x) := x$ for all $x \ge 0$, see Remark \ref{rem:vn_Lip_Hold}), \eqref{eq: borne phi n j} and \eqref{eq: borne psi n j ell} that we can find $C_{K}$ and $C_{K}'$, independent of $n$, such that
	\begin{align*}
		& \big| D  v^{n}(t+h^{4},   \xr_{t \wedge },\xr_{t} ) - D  v^{n}(t,  \xr, \xr_{t}) \big| \\
		&\le h^{-1} \big|v^{n}(t+h^{4}, \xr_{t \wedge },\xr_{t} +h) - v^{n}(t+h^{4}, \xr_{t \wedge },\xr_{t} ) - v^{n}(t, \xr, \xr_{t}+h)+ v^{n}(t, \xr,\xr_{t} ) \big|\\
		&\;\;+2C_{K}h \\
		&\le  h^{-1} \big| v^{n}(t+h^{4},( \xr  +\1_{\{t\}} h)_{t \wedge },  \xr_{t}+h)- v^{n}(t, \xr ,  \xr_{t}+h) \big|\\
		&\;\; +  h^{-1} \big| v^{n}(t, \xr,\xr_{t} )- v^{n}(t+h^{4}, \xr_{t \wedge },\xr_{t} ) \big|\\
 		&\;\;+ h^{-1} \big|v^{n}(t+h^{4}, \xr_{t \wedge },\xr_{t}+h)- v^{n}(t+h^{4},( \xr   +\1_{\{t\}} h)_{t \wedge },\xr_{t}+h) \big|\\
		&\;\;+2C_{K}h\\
		&\le 2C'_{K}h+C_{\beta} \lambda([t-|\pi^{n}|,t+h^{4}+|\pi^{n}|)),
	\end{align*} 
	for all $\xr \in D([0,T])$ and $t< t+h\le T$.
	
	\vspace{0.5em}

	 The rest of the proof is similar to the one of   \cite[Proposition 3.13]{bouchard2021c}: {(i)}  implies that $\nabla_{\xr}v$ is well-defined on a dense subset of $  \Theta$, including the (countably many) atoms of $\lambda$, while {(a)} and {(b)} allow to extend it in a continuous way and imply that the extension is actually the pointwise limit of $Dv^{n}$ on $\Theta$.  The required estimates are deduced from  {(a)} and {(b)} .
\end{proof}

\subsection{A semi-linear case}

	We now  consider a semi-linear PPDE {with generator} of the form
	\begin{equation} \label{eq:BSDE_F}
		F(t, \xr, y, z, \gamma) 
		~=~
		f(t, \xr, y,  \sigma^{\top} z ) + \mu(t, \xr) \cdot z + \frac12 \Tr \big[ \sigma \sigma^{\top} \gamma \big],
	\end{equation}
	for some non-anticipative functionals $f: \Theta \x \R \x \R^d \longrightarrow \R$ and  $(\mu, \sigma): \Theta \longrightarrow \R^d \x \S^d$. In this case, we can remove the structure condition \eqref{eq: hyp pi}. We make the following assumption on the coefficients. 
	
	\begin{Assumption} \label{eq:BSDE_FrechetCond}
		$\mathrm{(i)}$
		The functionals $\mu$ and $\sigma$ are  continuous.
		For each $t \in [0,T]$, the map $\xr\in D([0,T]) \longmapsto (\mu(t, \xr), \sigma(t, \xr))$ is Fr\'echet differentiable with derivative\footnote{We identify the Fr\'echet derivatives with the associated measures on $[0,T]$.} $(\lambda_{\mu}, \lambda_{\sigma}) (\cdot;t,\xr)$ at $\xr \in \DT$.
		
		\vspace{0.5em}
		
		\noindent $\mathrm{(ii)}$ The functional $g$ is Fr\'echet differentiable with derivative $\lambda_g(\cdot; \xr)$ at $\xr \in \DT$.
		The functional $f$ is continuous and non-anticipative, and for each $(t, y, z) \in [0,T] \x \R \x \R^d$, the map $\xr \longmapsto f(t, \xr, y,z)$ is Fr\'echet differentiable with derivative $\lambda_f(\cdot; t, \xr, y, z)$ at $\xr \in \DT$.
		Further, $(y, z){\in \R\x \R^{d }} \longmapsto f(t, \xr, y, z)$ is differentiable with  derivatives $(\partial_y f, \partial_z f)$ which are bounded and continuous in $(t, \xr, y, z)$.
		
		\vspace{0.5em}
		
		\noindent $\mathrm{(iii)}$ There exists some finite positive measure $\lambda$ on $[0,T]$, such that, for all $(t, \xr, y, z) \in [0,T] \x \DT \x \R \x \R^d$,
		$$
			\big| \lambda_{\mu} (\cdot; t, \xr) \big|
			+
			\big| \lambda_{\sigma} (\cdot; t, \xr) \big|
			+
			\big| \lambda_{f} (\cdot; t, \xr, y, z) \big|
			+
			\big| \lambda_{g} (\cdot; \xr)  \big|
			\ll \lambda(\cdot).
		$$

		\noindent $\mathrm{(iv)}$ 
		There exist constants $\alpha\in [0,1]$ and $C_{\alpha}>0$ such that 
		\begin{align*}
			&\left|\int_{0}^{T} \tilde \xr_{r} \lambda_{g}(dr;\xr) -\int_{0}^{T} \tilde \xr_{r} \lambda_{g}(dr;\xr')\right|\le C_{\alpha} \|\xr-\xr'\|^{\alpha} \|\tilde \xr\|,
			\\
			&\left|\int_{0}^{T} \tilde \xr_{r} \lambda_{\ell}(dr;s,\xr) -\int_{0}^{T} \tilde \xr_{r} \lambda_{\ell}(dr;s,\xr')\right|\le C_{\alpha} \|\xr-\xr'\|^{\alpha} \|\tilde \xr\| ,\;\ell\in \{b,\sigma \},
			\\
			&\left|\int_{0}^{T} \!\! \tilde \xr_{r} \lambda_{f}(dr;s,\xr,y,z) -\int_{0}^{T} \!\! \tilde \xr_{r} \lambda_{f}(dr;s,\xr',y',z')\right|\le C_{\alpha}\left(\|\xr-\xr'\|+|y-y'|+|z-z'|\right)^{\alpha} \|\tilde \xr\|, \\
			&\left|(\partial_{y},\partial_{z})f(s,\xr,y,z)-(\partial_{y},\partial_{z})f(s,\xr',y',z')\right|\le C_{\alpha}\left(\|\xr-\xr'\|+|y-y'|+|z-z'|\right)^{\alpha},
		\end{align*}
		{for all $\xr,\xr',\tilde \xr \in D([0,T])$ and $(s,y,z)\in [0,T]\x \R\x \R^{d}$.}
	\end{Assumption}

	Under Assumption \ref{eq:BSDE_FrechetCond},   \eqref{eq:PPDE}-\eqref{eq: cond bord} has a unique $\pi$-approximate viscosity solution $\vr$ (see Proposition \ref{prop : stoch representation}), which corresponds to the solution of a BSDE.
	We will then show that $\vr$ has a Dupire's vertical derivative which enjoys further continuity conditions. 
	In the following, we assume that we are given a fixed complete probability space, equipped with a Brownian motion $W$ together with the Brownian filtration.
	Let $S^2([t, T])$ denote the space of $\R^{d}$-valued predictable processes $X = (X_s)_{s \in [t,T]}$ such that $\E\big[ \sup_{t \le s \le T} |X_s|^{2} \big] < \infty$,
	 and $H^2([t,T])$ denote the space of $\R^{d}$-valued predictable processes $Z = (Z_s)_{s \in [t,T]}$ such that $\E \big[\int_{t}^{T}|Z_{s}|^{2}ds \big] < \infty$.

	\begin{Theorem}\label{thm : regul BSDE}
		$\mathrm{(i)}$ Let $F$ be given by \eqref{eq:BSDE_F}, and let Assumption \ref{eq:BSDE_FrechetCond} hold true. 
		Let $\vr$ be the unique $\pi$-approximate viscosity solution of \eqref{eq:PPDE}-\eqref{eq: cond bord}.
		Then, it admits a Dupire's vertical derivative $\nabla_{\xr} \vr$ on $[0,T)\x D([0,T])$ and there exists $C>0$ such that 
		\begin{equation*}  
			|\nabla_{\xr} \vr(t,\xr')- \nabla_{\xr} \vr(t,\xr)|
			~\le~
			C \|\xr'-\xr\|^{\alpha},
		\end{equation*}
		and 
		\begin{align*} 			
		|\nabla_{\xr} \vr(t',\xr_{t \wedge })- \nabla_{\xr} \vr(t,\xr)|
			~\le~
			C \big(|t'-t|^{\frac{\alpha}{2+2\alpha}}(1+\|\xr\|)+\lambda([t,t'))   \big),  
		\end{align*}
		for all $t\le t'\le T$ and $\xr,\xr'\in \DT$. 
		
		\vspace{0.5em}
		
		\noindent $\mathrm{(ii)}$
		Moreover, if $\lambda$ has finitely many atoms, then $\vr$ solves the stochastic path-dependent equation 
		\begin{align*}
			\vr(\cdot,X^{t,\xr}) 
			= 
			g \big( X^{t,\xr} \big)  
			+ 
			\! \int_{\cdot}^{T} \!\!\!\! f(u, X^{t,\xr}, \vr(u,X^{t,\xr}), [\sigma^{\top}\nabla_{\xr} \vr](u,X^{t,\xr})  ) du
			- 
			\! \int_{\cdot}^{T} \!\!\! [\nabla_{\xr} \vr\sigma] (u,X^{t,\xr}) dW_{u},
		 \end{align*}
		 in which $X^{t, \xr} \in S^2([0,T])$ is the unique solution of 
		$$
		X^{t,\xr}_s = \xr_{t \wedge s}+ \int_{t}^{t\vee s} \mu(u, X^{t,\xr}) du +\int_{t}^{t \vee s} \sigma(u, X^{t,\xr}) dW_{u},\;t\le s\le T.
	        $$
	\end{Theorem}

	We split the proof in different steps. 
	First, we will use the probabilistic representation of the PPDE (with generator $F$ in \eqref{eq:BSDE_F})
	in terms of forward-backward SDE to deduce the regularity of the $\pi$-approximate viscosity solution. 
	Namely, under  Assumption \ref{eq:BSDE_FrechetCond}, it follows from  Proposition \ref{prop : stoch representation} that the approximate viscosity solution $\vr$ is given by
	\begin{equation} \label{eq:vr2Y}
		\vr(t, \xr) ~=~ Y^{t, \xr}_t,
	\end{equation}
	where  $(Y^{t, \xr}, Z^{t, \xr}) \in S^2([t,T] )\x H^2([t,T])$ solves the BSDE
	$$
		Y^{t, \xr}_s = g \big( X^{t,\xr} \big)  + \int_{t\vee s}^{T} f(u, X^{t,\xr}, Y^{t,\xr}_{u}, Z^{t,\xr}_{u} ) du - \int_{t\vee s}^{T} Z^{t,\xr}_{u}\cdot dW_{u},
		~~s \in [t,T].
	$$
	Next, let us fix a path $\hat \xr \in \DT$, and then define a process $\big(\nabla^{t}_{\xr} X^{t,\xr}_{s}\big)_{s \in [0,T]}$ as the solution to the SDE
	$$
		\nabla^{t}_{\xr} X^{t,\xr}_{s}
		=
		\hat \xr_{s\wedge t}
		+
		\int_{t}^{s\vee t} \int_{t}^{u} \nabla^{t}_{\xr}X^{t,\xr}_{r} \lambda_{\mu}(dr;u, X^{t,\xr})du+\int_{t}^{s}\int_{t}^{u}\nabla^{t}_{\xr}  X^{t,\xr}_{r}\lambda_{\sigma}(dr;u, X^{t,\xr})dW_{u}.
	$$
	Namely, $\nabla^{t}_{\xr} X^{t,\xr}_{s}$ is the tangent process of $X^{t,\xr}$ in the direction of $\hat \xr$.
	We also define the processes $\big(\nabla^{t}_{\xr} Y^{t,\xr}_{s}, \nabla^{t}_{\xr} Z^{t,\xr}_{s} \big)_{s \in [t, T]}$ as the solution to the linear BSDE
	\begin{align*}
		\nabla^{t}_{\xr} Y^{t,\xr}_{s}
		=&
		\int_{0}^{T} \!\! \nabla^{t}_{\xr} X^{t,\xr}_{r}\lambda_{g}(dr;  X^{t,\xr})
		+
		\!\!\int_{s}^{T} \!\!\! \Big( \!\!
			\int_{t}^{u}\!\!  \nabla^{t}_{\xr} X^{t,\xr}_{r}  \lambda_{f}(dr; \Theta^{t,\xr}_{u}) 
			+
			A^{t,\xr}_{u}\cdot(\nabla^{t}_{\xr} Y^{t,\xr}_{u},\nabla^{t}_{\xr} Z^{t,\xr}_{u})
		\Big) du\\
		&-
		\int_{s}^{T}\nabla^{t}_{\xr} Z^{t,\xr}_{u}dW_{u},
	\end{align*}
	in which 
	\begin{equation} \label{eq:defA}
		\Theta^{t,\xr}_{s} := \big( s, X^{t,\xr},Y^{t,\xr}_{s} ,Z^{t,\xr}_{s} \big) ,
		~~\mbox{and}~
		A^{t,\xr}_s := \big( \partial_{y} f (\Theta^{t,\xr}_{s}), \partial_{z} f  (\Theta^{t,\xr}_{s}) \big).
	\end{equation}

	Under Assumption \ref{eq:BSDE_FrechetCond}, and by standard arguments (based on Burkholder-Davis-Gundy's inequality and Gronwall's lemma, see e.g. \cite{el1997backward}), 
	one has the following estimates on the processes $ \big(X^{t, \xr}, Y^{t, \xr}, Z^{t,\xr} \big)$ and $\big( \nabla^{t}_{\xr} X^{t,\xr}, \nabla^{t}_{\xr} Y^{t,\xr}, \nabla^{t}_{\xr} Z^{t,\xr} \big)$,
	whose proof is ommited.
	
	\begin{Lemma}
		Let Assumption \ref{eq:BSDE_FrechetCond} hold true.
		Then, for all $p \ge 1$, there exists a constant $C_p > 0$, such that, for all $t, t' \in [0,T]$ and $\xr, \xr' \in \DT$,
		\begin{align} \label{eq:EstimationXYZ}
			\E \big[ \|X^{t,\xr}\|^{p} \big]
			+ \sup_{s \in [0,T]} \E \big[ \big| Y^{t,\xr}_s \big|^{p} \big] 
			+ \E \Big[ \Big( \int_{0}^{T}|Z^{t,\xr}_{s}|^{2} \Big)^{\frac{p}2}ds \Big]
			~\le~
			C_{p}  \big(1+\|\xr_{t \wedge}\|^{p} \big),
		\end{align}
		\begin{align} \label{eq:EstimationContinuityXYZ}
			&
			\E \Big[ \big\|X^{t,\xr}-X^{t',\xr'} \big\|^{p} \Big]^{\frac1p} 
			+ \E \Big[ \big\| Y^{t,\xr}- Y^{t',\xr'} \big\|^{p} \Big]^{\frac1p} 
			+ \E \Big[ \Big(\int_{0}^{T} \big|Z^{t,\xr}_{s}-Z^{t',\xr'}_{s} \big|^{2}ds \Big)^{\frac{p}2} \Big]^{\frac1p} 
			\nonumber \\
			\le~&
			C_{p}  \Big( \int_{0}^{t\vee t'}|\xr_{s}-\xr'_{s}|\lambda(ds) + (1+\|\xr\|+\|\xr'\|) |t-t'|^{\frac12} \Big),
		\end{align}
		and
		\begin{align} \label{eq:EstimationDXYZ} 
			\E \Big [ \big\| \nabla^{t}_{\xr}X^{t,\xr} \big\|^{p} \Big]^{\frac1p}
			+
			\E \Big[ \big\| \nabla^{t}_{\xr}Y^{t,\xr} \big\|^{p} \Big]^{\frac1p}
			+
			\E \Big[ \Big(\int_{0}^{T} \big|\nabla^{t}_{\xr}Z^{t,\xr}_{s} \big|^{{2}} ds \Big)^{\frac{p}2} \Big]
			\le C_{p} .
		\end{align}
	\end{Lemma}

	Next, for each $\delta > 0$, we denote
	$$
		\big( \nabla^{t, \delta}_{\xr} X^{t,\xr}_{s}, \nabla^{t, \delta}_{\xr} Y^{t,\xr}, \nabla^{t, \delta}_{\xr} Z^{t,\xr} \big)
		:=
		\delta^{-1} \Big(  X^{t,\xr+\delta \hat \xr}_{s}-X^{t,\xr}_{s},  Y^{t,\xr+\delta \hat \xr}-Y^{t,\xr} ,  Z^{t,\xr+\delta \hat \xr}-Z^{t,\xr} \Big).
	$$

	\begin{Lemma} \label{lem:XY_tangent}
		Let Assumption \ref{eq:BSDE_FrechetCond} hold true.
		Then, for all $(t,\xr)\in [0,T]\x \CT$,
		\begin{equation} \label{eq:nabla_delta_XY_cvg}
			\lim_{\delta \searrow 0} 
			\E \Big[ \big\| \nabla^{t, \delta}_{\xr} X^{t,\xr}  -  \nabla^{t}_{\xr} X^{t,\xr} \big\|^2 \Big]
			=0
			~~\mbox{and}~~
			\lim_{\delta \searrow 0}
			\Big| \nabla^{t, \delta}_{\xr} Y^{t,\xr}_t - \nabla^{t}_{\xr} Y^{t,\xr}_t \Big|
			= 0.
		\end{equation}
	\end{Lemma}
	\proof 
	For ease of notation, let us omit the superscript $(t, \xr)$ in the notations of the processes $\big(X^{t, \xr+\delta \hat \xr}, X^{t, \xr}, \nabla^{t, \delta}_{\xr} X^{t,\xr}, \nabla^{t}_{\xr} X^{t,\xr} \big)$ and write them as $(X^{\delta}, X, \nabla^{t, \delta}_{\xr} X, \nabla^{t}_{\xr} X)$.
	
	\vspace{0.5em}
	
	\noindent $\mathrm{(i)}$ 
	We first notice that $(\nabla^{t, \delta}_{\xr} X_s)_{s \in [t,T]} $ satisfies
	\begin{align*}
		\nabla^{t, \delta}_{\xr} X_s 
		&=
		I_{d}
		+ 
		\int_{t}^{s} \!\!\big( \mu(u, X^{\delta}) - \mu(u, X)\big) du 
		+
		\int_{t}^{t \vee s} \!\! \big( \sigma(u, X^{\delta}) - \sigma(u, X)\big) dW_{u} \\
		&=
		I_{d}
		+
		\int_{t}^{s} \int_{t}^{u} \big( \nabla^{t,\delta}_{\xr}  X_{r} \big) \lambda_{\mu}^{\delta} (dr;u)du
		+
		\int_{t}^{s}\int_{t}^{u} \big( \nabla^{t, \delta}_{\xr}  X_{r} \big) \lambda_{\sigma}^{\delta}(dr;u)dW_{u},
	\end{align*}
	where {$I_{d}$ is the identity matrix and}, as $\delta \searrow 0$,
	\begin{equation} \label{eq:lambda_delta_cvg}
		\big(\lambda_{\mu}^{\delta}(\cdot, u), \lambda^{\delta}_{\sigma}(\cdot, u) \big) 
		\longrightarrow
		\big(\lambda_{\mu}(dr;u), \lambda_{\sigma}(dr;u) \big)
		:=
		\big(\lambda_{\mu}(dr;u, X), \lambda_{\sigma}(dr;u, X) \big).
	\end{equation}
	Applying It\^o formula on $\big( \nabla^{t, \delta}_{\xr} X_s  - \nabla^{t}_{\xr} X_s \big)^2$ and then taking expectation,
	it follows by standard arguments that, for some constant $C > 0$ independent of $\delta$, and a constant $C_{\delta} > 0$,
	$$
		\E \Big[ \big\| \nabla^{t,\delta}_{\xr} X_{s \wedge} - \nabla^{t}_{\xr} X_{s \wedge} \big\|^2 \Big]
		~\le~
		C \int_t^s \E \Big[ \big\| \nabla^{t,\delta}_{\xr} X_{r \wedge} - \nabla^{t}_{\xr} X_{r \wedge} \big\|^2 \Big] dr + C_{\delta},
	$$
	where, by the moment estimation in \eqref{eq:EstimationXYZ} and the convergence in \eqref{eq:lambda_delta_cvg},
	\begin{align*}
		C_{\delta} 
		~\le~&
		C \int_t^T  \E \Big[ \big|\nabla_{\xr}^{t, \delta} X_u - \nabla_{\xr}^{t} X_u \big|
						\Big| \int_t^T |\nabla_{\xr}^{t} X_r | \big( \lambda^{\delta}_{\mu} -\lambda_{\mu} \big) \Big| (dr, u) 
			\Big] du \\
		&+~
		C \int_t^T \E \Big[ 
				\Big( \int_t^T \nabla^t_{\xr} X_r  \big( \lambda^{\delta}_{\mu} -\lambda_{\mu} \big)(dr, u) \Big)^2
			\Big] du
		~\longrightarrow~ 0.
	\end{align*}
	By Gronwall's lemma, it follows that the first convergence result in \eqref{eq:nabla_delta_XY_cvg} holds true.
	
	\vspace{0.5em}
	
	\noindent $\mathrm{(ii)}$ We notice that $\big(\nabla^{t, \delta}_{\xr} Y, \nabla^{t, \delta}_{\xr} Z \big)$ satisfies
	\begin{align*}
		\nabla^{t, \delta}_{\xr} Y_s 
		=~&
		g ( X^{\delta}) - g(X)  + \int_{s}^{T}\!\! \Big( f(u, X^{\delta}, Y^{\delta}_{u}, Z^{\delta}_{u} ) - f(u, X, Y_{u}, Z_{u} ) \Big) du 
		- 
		\int_{s}^{T}\!\! \nabla^{t, \delta}_{\xr} Z_{u}dW_{u} \\
		=~&
		\int_{t}^{T} \!\! \big( \nabla^{t,\delta}_{\xr} X_{r} \big) ~\lambda^{\delta}_{g}(dr)
		+
		\int_{s}^{T} \!\!
			\Big(\!
				\int_{t}^{u} \!\! \big( \nabla^{t,\delta}_{\xr} X_r \big)  \lambda^\delta_{f}(dr; {u}) 
				+ A^{\delta}_u \cdot \big( \nabla^{t,\delta}_{\xr} Y_{u}, \nabla^{t,\delta}_{\xr} Z_{u} \big)
			\Big) du \\
		& -
		\int_{s}^{T} \nabla^{t,\delta}_{\xr} Z_{u}dW_{u},
	\end{align*}
	where 
	$$
		\big(\lambda_{f}^{\delta}(\cdot, u), \lambda^{\delta}_{g}(\cdot) \big) 
		\longrightarrow
		\big(\lambda_{f}(dr;u), \lambda_{g}(dr) \big)
		:=
		\big(\lambda_{f}(dr;u, X, Y, Z), \lambda_{g}(dr; X) \big),
		~\mbox{as}~ \delta \searrow 0,
	$$
	and
	$$
		 A^{\delta}_u \cdot \big( \nabla^{t,\delta}_{\xr} Y_{u}, \nabla^{t,\delta}_{\xr} Z_{u} \big)
		~:=~
		\delta^{-1} \Big( f\big(u, X^{\delta}, Y^{\delta}_{u}, Z^{\delta}_{u} \big) - f \big(u, X^{\delta}, Y, Z \big)  \Big).
	$$
	Recalling \eqref{eq:defA}, it follows that, for all $p \ge 1$,
	$$
		\E \Big[ \int_t^T | A^{\delta}_u - A_u |^p du  \Big]
		~\longrightarrow~
		0,
		~\mbox{as}~ \delta \searrow 0.
	$$
	One can then conclude by the stability result for BSDEs (see e.g.   \cite[Proposition 2.1]{el1997backward}
	or  \cite[Theorem 4.1 and Remark 4.1]{bouchard2018unified}).
 	\qed

	\begin{proof}[Proof of Theorem \ref{thm : regul BSDE}]
	$\mathrm{(i)}$ Let us first fix $\hat \xr := \mathbf{1}_{\{[t,T]\}}$.
	By \eqref{eq:vr2Y} {and \eqref{eq:EstimationContinuityXYZ}}, we observe that {$ \vr$} is locally $1/2$-H\"older in time and Lipschitz in space.
	{Moreover}, by Lemma \ref{lem:XY_tangent}, $\nabla_{\xr} \vr(t, \xr) = \nabla^{t}_{\xr} Y^{t,\xr}_t$.
	
	\vspace{0.5em}

	Thus, it is enough to study the regularity of $\nabla^{t}_{\xr} Y^{t,\xr}_t$ in $(t, \xr)$.
	By direct computation, one has {(using the notation \eqref{eq:defA})}
	\begin{align*}
		\Delta(t,\xr,\xr')
		~:=~&
		\big| \nabla^{t}_{\xr} Y^{t,\xr}_{t}- \nabla^{t}_{\xr} Y^{t,\xr'}_{t} \big|^{2} \\
		~\le~& 
		\E \Big [ \Big|\int_{t}^{T} \nabla^{t}_{\xr} X^{t,\xr}_{r}\lambda_{g}(dr;  X^{t,\xr}) -\int_{t}^{T} \nabla^{t}_{\xr} X^{t,\xr'}_{r}\lambda_{g}(dr;  X^{t,\xr'}) \Big|^{2} \Big]\\
		&+~
		\E \Big [\int_{t}^{T} \Big|\int_{t}^{s} \nabla^{t}_{\xr} X^{t,\xr}_{r}  \lambda_{f}(dr;\Theta^{t,\xr}_{s})-\int_{t}^{s} \nabla^{t}_{\xr} X^{t,\xr'}_{r}  \lambda_{f}(dr;\Theta^{t,\xr'}_{s}) \Big|^{2} ds \Big]\\
		&+ ~
		\E \Big[\int_{t}^{T} \Big|(A^{t,\xr}_{s}- A^{t,\xr'}_{s})\cdot(\nabla^{t}_{\xr} Y^{t,\xr}_{s},\nabla^{t}_{\xr} Z^{t,\xr}_{s})   \Big|^{2} ds \Big].
	 \end{align*}	
	Further, using Assumption \ref{eq:BSDE_FrechetCond} and   Grownwall's lemma, one can find a constant $C >0$ such that
	\begin{align*}
		\E \Big[
			\big\| \nabla^{t}_{\xr} X^{t,\xr}-\nabla^{t}_{\xr} X^{t,\xr'} \big\|^{2} 
		\Big]^{\frac12}
		+
		\Delta(t,\xr,\xr')^{\frac12}
		~\le~
		C \|\xr-\xr'\|^{\alpha}.
	\end{align*} 
	It follows that $\nabla_{\xr}\vr$ is $\alpha$-H\"older in space.
	
	\vspace{0.5em}
	
	We now repeat the argument used at the end of the proof of Theorem \ref{thm: C1alpha avec pji}. In view of \eqref{eq:EstimationContinuityXYZ}, we can find $C, C'$ such that 
	\begin{align*}
		& \big| \nabla_{\xr}\vr(t+h^{2+2\alpha},   \xr_{t \wedge } ) - \nabla_{\xr}\vr(t,  \xr, \xr_{t}) \big| \\
		\le~& h^{-1} \big|\vr(t+h^{2+2\alpha}, \xr_{t \wedge }+\1_{\{t+h^{2+2\alpha}\}}h) - \vr(t+h^{2+2\alpha}, \xr_{t \wedge }  ) - \vr(t, \xr  \oplus_{t} h)+ \vr(t, \xr ) \big|
			+2Ch^{\alpha}\\
		\le~&  h^{-1} \big| \vr(t+h^{2+2\alpha},( \xr + \1_{\{t\}} h)_{t \wedge })- \vr(t, \xr \oplus_{t} h) \big|
		+  h^{-1} \big| \vr(t, \xr)- \vr(t+h^{2+2\alpha}, \xr_{t \wedge }  ) \big|\\
 		&\;\;+ h^{-1} \big| \vr(t+h^{2+2\alpha}, \xr_{t \wedge }+\1_{\{t+h^{2+2\alpha}\}}h)- \vr(t+h^{2+2\alpha},( \xr  +\1_{\{t\}} h)_{t \wedge }) \big|
		+2Ch^{\alpha}\\
		\le~& C'\left(h^{\alpha}(1+\|\xr\|+h)+\lambda([t,t+h)) \right),
	\end{align*} 
	for all $\xr \in D([0,T])$ and $t< t+h\le T$. 

	\vspace{0.5em}

	\noindent $\mathrm{(ii)}$ Assume now that $\lambda$ has finitely many atoms at times $(s_{i})_{i\le I}\subset [0,T]$ with $s_{i}<s_{i+1}$ for $i+1\le I$. Then,  \cite[Theorem 2.5]{bouchard2021c} can be applied on each time interval $[s_{i},s_{i+1})$. Since $\vr(\cdot,X^{t,\xr})$ and $\int_{t}^{\cdot} [\nabla_{\xr} \vr\sigma](s,X^{t,\xr})dW_{s}$ are continuous, this implies that we can find a continuous  (and therefore predictable) orthogonal process $A$, see \cite[Definition 2.3]{bouchard2021c} such that 
$$
\vr(\cdot,X^{t,\xr})=\vr(t,X^{t,\xr})+\int_{t}^{\cdot} [\nabla_{\xr} \vr\sigma](s,X^{t,\xr})dW_{s}+A\;.
$$
In particular $[A,W]=0$.   Thus, the fact that $Y^{t,\xr}=\vr(\cdot,X^{t,\xr})$ implies that $(Z^{t,\xr})^{\top}= [\nabla_{\xr} \vr\sigma](\cdot,X^{t,\xr})$ $dt\times d\P$ on $[t,T]$ by uniqueness of the co-quadratic variation of   semimartingales.
 	\end{proof}

\bibliographystyle{plain}

\begin{thebibliography}{}

\end{thebibliography}


\begin{thebibliography}{10}

\bibitem{bouchard2021c}
Bruno Bouchard, Gr{\'e}goire Loeper, and Xiaolu Tan.
\newblock A $C^{0, 1}$ -functional It\^{o}'s formula and its applications in
  mathematical finance.
\newblock {\em arXiv preprint arXiv:2101.03759}, 2021.

\bibitem{bouchard2018unified}
Bruno Bouchard, Dylan Possama{\"\i}, Xiaolu Tan, and Chao Zhou.
\newblock A unified approach to a priori estimates for supersolutions of bsdes
  in general filtrations.
\newblock {\em Annales de l'Institut Henri Poincar{\'e}, Probabilit{\'e}s et
  Statistiques}, 54(1):154--172, 2018.

\bibitem{cont2013functional}
Rama Cont and David-Antoine Fourni{\'e}.
\newblock Functional It{\^o} calculus and stochastic integral representation of
  martingales.
\newblock {\em The Annals of Probability}, 41(1):109--133, 2013.


\bibitem{cosso2021path}
Andrea Cosso, Fausto Gozzi, Mauro Rosestolato, and Francesco Russo.
\newblock Path-dependent Hamilton-Jacobi-Bellman equation: Uniqueness of Crandall-Lions viscosity solutions.
\newblock {\em arXiv preprint arXiv:2107.05959}, 2021.


\bibitem{cossorusso14}
Andrea Cosso and Francesco Russo.
\newblock A regularization approach to functional it\^{o} calculus and
  strong-viscosity solutions to path-dependent pdes.
\newblock {\em arXiv preprint arXiv:1401.5034}, 2014.

\bibitem{cosso2019crandall}
Andrea Cosso and Francesco Russo.
\newblock Crandall-lions viscosity solutions for path-dependent pdes: The case
  of heat equation.
\newblock {\em arXiv preprint arXiv:1911.13095}, 2019.


\bibitem{CrandallIshiiLions}
Michael~G. Crandall, Hitoshi Ishii, and Pierre-Louis Lions.
\newblock User's guide to viscosity solutions of second order partial
  differential equations.
\newblock {\em Bull. Amer. Math. Soc. (N.S.)}, 27(1):1--67, 1992.

\bibitem{dupireito}
Bruno Dupire.
\newblock Functional It\^{o} calculus.
\newblock {\em Portfolio Research Paper}, 04, 2009.

\bibitem{ekren2014viscosity}
Ibrahim Ekren, Christian Keller, Nizar Touzi, and Jianfeng Zhang.
\newblock On viscosity solutions of path dependent pdes.
\newblock {\em The Annals of Probability}, 42(1):204--236, 2014.

\bibitem{ekren2016viscosity}
Ibrahim Ekren, Nizar Touzi, and Jianfeng Zhang.
\newblock Viscosity solutions of fully nonlinear parabolic path dependent pdes:
  Part i.
\newblock {\em The Annals of Probability}, 44(2):1212--1253, 2016.

\bibitem{ekren2016viscosity2}
Ibrahim Ekren, Nizar Touzi, and Jianfeng Zhang.
\newblock Viscosity solutions of fully nonlinear parabolic path dependent pdes:
  Part ii.
\newblock {\em Annals of Probability}, 44(4):2507--2553, 2016.

\bibitem{ekren2016pseudo}
Ibrahim Ekren and Jianfeng Zhang.
\newblock Pseudo-markovian viscosity solutions of fully nonlinear degenerate
  ppdes.
\newblock {\em Probability, Uncertainty and Quantitative Risk}, 1(1):1--34,
  2016.

\bibitem{el1997backward}
Nicole El~Karoui, Shige Peng, and Marie~Claire Quenez.
\newblock Backward stochastic differential equations in finance.
\newblock {\em Mathematical finance}, 7(1):1--71, 1997.

\bibitem{hairer2015loss}
Martin Hairer, Martin Hutzenthaler,  and Arnulf Jentzen. 
\newblock{Loss of regularity for Kolmogorov equations},
\newblock {\em The Annals of Probability}, {43}(2), {468--527}, {2015}.


\bibitem{lieberman1996second}
Gary~M. Lieberman.
\newblock {\em Second order parabolic differential equations}.
\newblock World scientific, 1996.

\bibitem{peng2015g}
Shige Peng and Yongsheng Song.
\newblock G-expectation weighted sobolev spaces, backward sde and path
  dependent pde.
\newblock {\em Journal of the Mathematical Society of Japan}, 67(4):1725--1757,
  2015.

\bibitem{peng2016bsde}
Shige Peng and Falei Wang.
\newblock Bsde, path-dependent pde and nonlinear feynman-kac formula.
\newblock {\em Science China Mathematics}, 59(1):19--36, 2016.

\bibitem{pham2014two}
Triet Pham and Jianfeng Zhang.
\newblock Two person zero-sum game in weak formulation and path dependent
  bellman--isaacs equation.
\newblock {\em SIAM Journal on Control and Optimization}, 52(4):2090--2121,
  2014.

\bibitem{possamai2020zero}
Dylan Possama{\"\i}, Nizar Touzi, and Jianfeng Zhang.
\newblock Zero-sum path-dependent stochastic differential games in weak
  formulation.
\newblock {\em The Annals of Applied Probability}, 30(3):1415--1457, 2020.

\bibitem{ren2017comparison}
Zhenjie Ren, Nizar Touzi, and Jianfeng Zhang.
\newblock Comparison of viscosity solutions of fully nonlinear degenerate
  parabolic path-dependent pdes.
\newblock {\em SIAM Journal on Mathematical Analysis}, 49(5):4093--4116, 2017.

\bibitem{soner2012wellposedness}
Halil Mete Soner, Nizar Touzi, and Jianfeng Zhang.
\newblock Wellposedness of second order backward sdes.
\newblock {\em Probability Theory and Related Fields}, 153(1-2):149--190, 2012.

\bibitem{zhou2020viscosity}
Jianjun Zhou.
\newblock Viscosity solutions to second order path-dependent Hamilton-Jacobi-Bellman equations and applications.
\newblock {\em arXiv preprint arXiv:2005.05309}, 2020.

\end{thebibliography}
\def\cprime{$'$} \def\cprime{$'$}

\end{document}